\documentclass[11pt,reqno]{amsart}
\usepackage[utf8]{inputenc}

\usepackage{amssymb, amsmath, amsthm}
\usepackage{hyperref}
\usepackage[alphabetic,lite]{amsrefs}
\usepackage{verbatim}
\usepackage{amscd}   
\usepackage[all]{xy} 
\usepackage{youngtab} 
\usepackage{young} 
\usepackage{ytableau}
\usepackage{tikz}
\usepackage{ mathrsfs }
\usepackage{cases}
\usepackage{array}
\usepackage{cellspace}
\usepackage{calligra,mathrsfs}
\usepackage{bm}
\usepackage{graphicx}
\usepackage{rank-2-roots}
\usepackage{float}

\newcommand{\defi}[1]{{\bf\upshape\sffamily #1}}
\DeclareMathOperator{\ShHom}{\mathscr{H}\text{\kern -3pt {\calligra\large om}}\,}

\renewcommand{\a}{\alpha}

\newcommand{\bw}{\bigwedge}

\def\kk{{\mathbf k}}
\renewcommand{\ll}{\lambda}

\newcommand{\onto}{\twoheadrightarrow}
\newcommand{\oo}{\otimes}
\newcommand{\pd}{\partial}

\newcommand{\s}{\sigma}

\newcommand{\bbs}{\mathbb{S}}

\newcommand{\Ext}{\operatorname{Ext}}

\newcommand{\GL}{\operatorname{GL}}
\newcommand{\Hom}{\operatorname{Hom}}

\newcommand{\Sym}{\operatorname{Sym}}

\newcommand{\rank}{\operatorname{rank}}

\newcommand{\bbw}{\mathbb{W}}

\newcommand{\chr}{\operatorname{char}}

\newcommand{\coker}{\operatorname{coker}}
\renewcommand{\det}{\operatorname{det}}

\renewcommand{\ker}{\operatorname{ker}}

\newcommand{\sgn}{\operatorname{sgn}}

\newcommand{\bb}[1]{\mathbb{#1}}

\renewcommand{\rm}[1]{\textrm{#1}}
\newcommand{\mc}[1]{\mathcal{#1}}
\newcommand{\mf}[1]{\mathfrak{#1}}

\newcommand{\op}[1]{\operatorname{#1}}

\newcommand{\ul}[1]{\underline{#1}}

\def\PP{{\mathbf P}}
\def\lra{\longrightarrow}

\newtheorem{theorem}{Theorem}[section]
\newtheorem*{theorem*}{Theorem}
\newtheorem*{problem*}{Problem}
\newtheorem{lemma}[theorem]{Lemma}

\newtheorem{proposition}[theorem]{Proposition}
\newtheorem{corollary}[theorem]{Corollary}
\newtheorem*{corollary*}{Corollary}

\newtheorem*{main-thm*}{Main Theorem}
\newtheorem*{stable-flag*}{Stable Cohomology Theorem on Flag Varieties}
\newtheorem*{stable-dd*}{Stable Cohomology Calculation for $\ll=(-d,d)$}
\newtheorem*{stable-recursion-hooks*}{Stable Cohomology Recursion for $\ll=(-a-b,a,1^b)$}
\newtheorem*{stable-proj*}{Stable Cohomology Theorem on Projective Space}
\newtheorem*{stable-hooks*}{Stable Cohomology Calculation for Hooks}
\newtheorem*{stable-twocolumn-hooks*}{Duality Theorem for 2-Column Partitions and Hooks}
\newtheorem*{stable-truncated-pows*}{Stable Cohomology Calculation for Truncated Powers}
\newtheorem*{stable-vanishing*}{Polynomial Functors with Vanishing Stable Cohomology}
\newtheorem*{vanishing-Koszul*}{Vanishing Theorem for Finite Length Koszul Modules}

\theoremstyle{definition}

\newtheorem*{definition*}{Definition}

\newtheorem{example}[theorem]{Example}

\theoremstyle{remark}
\newtheorem{remark}[theorem]{Remark}
\newtheorem*{remark*}{Remark}

\numberwithin{equation}{section}

\tikzset{
  treenode/.style = {align=center, inner sep=0pt, text centered,solid,thin,
    font=\sffamily},
  arn_n/.style = {treenode, circle, white, font=\sffamily\bfseries, draw=black,
    fill=black, text width=.5em},
  arn_nl/.style = {treenode, circle, white, font=\sffamily\bfseries, draw=black,
    fill=black, text width=1.5em},  
  arn_r/.style = {treenode, circle, red, draw=red, 
    text width=.5em, very thick},
  arn_v/.style = {treenode, circle, black, font=\sffamily\bfseries, draw=black, text width=1.2em},
  arn_x/.style = {treenode, rectangle, draw=black,
    minimum width=.5em, minimum height=0.5em},
  dott/.style={edge from parent/.style={dotted, very thick,circle,draw}},
  emph/.style={edge from parent/.style={dashed, very thick,circle,draw}},
  norm/.style={edge from parent/.style={solid,thin,circle,draw}}
}
\makeatletter
\def\labelbox#1{%
  \hbox{%
    \setbox\z@=\hbox{$\m@th\labelstyle{\,#1\,}$}%
    \setbox\tw@=\hbox{$\m@th\labelstyle\,$}%
    \dimen@=\ht\z@ \advance\dimen@ by \wd\tw@ \ht\z@=\dimen@
    \dimen@=\dp\z@ \advance\dimen@ by \wd\tw@ \dp\z@=\dimen@
    \box\z@
  }%
}
\makeatother
\begin{document}

\title{Stable sheaf cohomology on flag varieties}

\author{Claudiu Raicu}
\address{Department of Mathematics, University of Notre Dame, 255 Hurley, Notre Dame, IN 46556\newline
\indent Institute of Mathematics ``Simion Stoilow'' of the Romanian Academy}
\email{craicu@nd.edu}

\author{Keller VandeBogert}
\address{Department of Mathematics, University of Kentucky, 719 Patterson Office Tower, Lexington, KY 40506}
\email{keller.v@uky.edu}


\date{\today}


\begin{abstract} 
 We prove an effective stabilization result for the sheaf cohomology groups of line bundles on flag varieties $Fl_n$ parametrizing complete flags in $\kk^n$, as well as for the sheaf cohomology groups of polynomial functors applied to the cotangent sheaf $\Omega$ on projective space. In characteristic zero, these are natural consequences of the Borel--Weil--Bott theorem, but in characteristic $p>0$ they are non-trivial. Unlike  many important contexts in modular representation theory, where the prime characteristic $p$ is assumed to be large relative to $n$, in our study we fix $p$ and let $n\lra\infty$. We illustrate the general theory by providing explicit stable cohomology calculations in a number of cases of interest.  Our examples yield cohomology groups where the number of indecomposable summands has super-polynomial growth, and also show that the cohomological degrees where non-vanishing occurs do not form a connected interval. In the case of polynomial functors of $\Omega$, we prove a K\"unneth formula for stable cohomology, and show the invariance of stable cohomology under Frobenius, which combined with the Steinberg tensor product theorem yields calculations of stable cohomology for an interesting class of simple polynomial functors arising in the work of Doty. The results in the special case of symmetric powers of $\Omega$ provide a nice application to commutative algebra, yielding a sharp vanishing result for Koszul modules of finite length in all characteristics.
\end{abstract}

\maketitle

\section{Introduction}\label{sec:intro}

For $\kk$ an algebraically closed field, we consider the \defi{flag variety} $Fl_n$ parametrizing complete flags of subspaces $0\subset V_1\subset\cdots\subset V_{n-1} \subset \kk^n$, where $\dim(V_i)=i$. A fundamental question at the confluence of algebraic geometry and representation theory is to describe the cohomology of line bundles on $Fl_n$. If $\chr(\kk)=0$, the answer to this question is given by the Borel--Weil--Bott theorem, but the situation over fields of positive characteristic remains more mysterious. There have been a number of outstanding results established in this context, such as the Kempf vanishing theorem \cites{kempf-van,kempf,haboush,andersen-frob}, Andersen's proof of the Strong Linkage Principle \cite{And-str-link} (see also \cites{car-lusz,verma,wong,doty1989strong}), or his characterization of the (non-)vanishing of $H^1$ \cite{andersen}, but a general understanding of the structure of cohomology, or at the very least of its (non-)vanishing behavior, remains elusive. The goal of this paper is to study the asymptotic behavior of cohomology (as $n\lra\infty$) and prove effective stabilization results, as well as provide a series of explicit calculations of (stable) cohomology.

We identify $Fl_n = \GL_n/B$, where $B$ denotes the Borel subgroup of upper triangular matrices, and consider $\GL_n$-equivariant line bundles, which are parametrized by weights $\ll$ in $\bb{Z}^n$. We write $\mc{O}_{Fl_n}(\ll)$ for the line bundle corresponding to $\ll\in\bb{Z}^n$, and write $H^j(Fl_n,\mc{O}_{Fl_n}(\ll))$ for its cohomology groups, which are representations of~$\GL_n$. If $\ll\in\bb{Z}^r$ then we write $|\ll|=\ll_1+\cdots+\ll_r$ and
\begin{equation}\label{eq:ll-padded}
\ll^{[n]} = (\ll_1,\cdots,\ll_r,0,\cdots,0)\in\bb{Z}^n\quad \text{for}\quad n\geq r.
\end{equation}

\begin{stable-flag*}
 If $\ll\in\bb{Z}^r$ and $j\geq 0$, then there exists a polynomial functor $\mc{P}^j_{\ll}$ of degree $|\ll|$ such that
 \[H^j\left(Fl_n,\mc{O}_{Fl_n}\left(\ll^{[n]}\right)\right) = \mc{P}^j_{\ll}(\kk^n)\quad\text{ for }n\gg 0.\]
 In particular:
 \begin{itemize}
     \item If $|\ll|<0$ then $\mc{P}^j_{\ll}=0$ and we get vanishing of cohomology.
     \item If $|\ll|=0$ then $\mc{P}^j_{\ll}(\kk^n)$ is a finite dimensional vector space with a trivial $\GL_n$-action. We denote it by~$H^j_{st}(\ll)$, refer to it as the \defi{stable cohomology} of $\mc{O}(\ll)$, and write $h^j_{st}(\ll)$ for its dimension.
 \end{itemize}
\end{stable-flag*}

In Section~\ref{sec:stab-flag-coh} we give a more precise result that includes an effective bound for the stabilization result (see Theorem~\ref{thm:coh-lamn-poly}, Corollary~\ref{cor:stab-coh=poly-functor}). When $j=0$ the theorem is classical: $\mc{P}_{\ll}^0\neq 0$ if and only if $\ll_1\geq\ll_2\geq\cdots\geq \ll_r\geq 0$ (we call such $\ll$ a \defi{partition}), in which case $\mc{P}_{\ll}^0=\bb{S}_{\ll}$ is the \defi{Schur functor} corresponding to $\ll$. If $\ll=(d)$ then
\[ H^0\left(Fl_n,\mc{O}_{Fl_n}\left(d,0,\cdots,0\right)\right)=\Sym^d(\kk^n),\]
and if $\ll=(1,\cdots,1)=(1^r)$ then
\[ H^0\left(Fl_n,\mc{O}_{Fl_n}\left(1,\cdots,1,0,\cdots,0\right)\right)=\bw^r\kk^n.\]
Another easy case of the stabilization result is when $\ll=(-r,1,\cdots,1)=(-r,1^r)\in\bb{Z}^{r+1}$, in which case the only non-vanishing stable cohomology appears in degree $j=r$, and it is a vector space of dimension 
\[h^r_{st}(-r,1^r)=1.\]
In all the examples so far, the results are independent of the characteristic of $\kk$, but this is no longer true for general weights $\ll$. We illustrate this when $\ll=(-d,d)$ with $d\geq 2$. Already in this case it is easy to see that for small values of $n$ the resulting cohomology groups are not polynomial representations: if $n=2$ then $Fl_2\simeq\PP^1$ and $\mc{O}_{Fl_2}(-d,d) = (\bw^2\kk^2)^{\oo d} \oo \mc{O}_{\PP^1}(-2d)$, whose non-zero cohomology is
\[ H^1\left(Fl_2,\mc{O}_{Fl_2}(-d,d)\right) = \left(\bw^2\kk^2\right)^{\oo d-1} \oo (\Sym^{2d-2}\kk^2)^{\vee}.\]
If $\chr(\kk)=0$ and $d\geq 2$, then $h^j_{st}(-d,d)=0$ for all $j$, but the result is significantly more subtle when $\chr(\kk)=p>0$. In order to describe it, we first enumerate the non-negative integers $\equiv 0,1$ (mod $p$) as $0,1,p,p+1,2p,2p+1,\cdots$. For $m$ in this list, we write $|m|_p$ for its position, and call $|m|_p$ the \defi{$p$-index} of $m$:
\begin{equation}\label{eq:def-mp-index} \text{if }m=pa+b,\text{ with }b\equiv 0,1\text{ (mod }p)\text{ and }0\leq b<p,\text{ then }|m|_p = 2a+b.
\end{equation}
If $p=2$ then $|m|_p=m$, but when $p=3$ we get the following values of the $p$-index:
\[
\begin{array}{c|c|c|c|c|c|c|c|c|c|c|c|c|c}
m & 0 & 1 & 2 & 3 & 4 & 5 & 6 & 7 & 8 & 9 & 10 & 11 & \cdots \\ \hline
|m|_p & 0 & 1 &    &  2 & 3  &  & 4 & 5 &   & 6  &  7  &   & \cdots \\
\end{array}
\]
For a tuple $\ul{a} = (a_0,\cdots,a_k)$ of non-negative integers, with $a_i\equiv 0,1\text{ mod }p$, we write
\[|\ul{a}|_p = \sum_{i=0}^k |a_i|_p,\]
and we define
\begin{equation}\label{eq:def-Apd}
A_{p,d} = \left\{ \ul{a}=(a_0,\cdots,a_k) : \sum_{i=0}^k a_i\cdot p^i = d,\ a_i\geq 0,\ a_i\equiv 0,1\text{ (mod $p$)}\right\}.
\end{equation}
The following theorem will be proved in an equivalent form in Section~\ref{subsec:coh-Sym-Omega} (see Theorem~\ref{thm:gen-fun-stab-coh-SymOmega}).

\begin{stable-dd*}
 Suppose that $d\geq 2$ and that $\chr(\kk)=p>0$. We have
 \begin{equation}\label{eq:stab-O(-d,d)} 
 \sum_{j\geq 0}h^j_{st}(-d,d)\cdot t^j = \sum_{\ul{a}\in A_{p,d}} t^{|\ul{a}|_p}.
 \end{equation}
\end{stable-dd*}

\begin{example}\label{ex:stab-dd}
    We illustrate the stable cohomology calculation for $d=6$ (see also Example~\ref{ex:cohSyma-small-vals}). It follows from \eqref{eq:def-Apd} that $A_{p,d}=\emptyset$ if $d\not\equiv 0,1\text{ (mod $p$)}$, so $h^j_{st}(-6,6)=0$ for all $j$ if $p\neq 2,3,5$. When $p=2$ we get
    \[ A_{2,6} = \{(0,1,1),(2,0,1),(0,3),(2,2),(4,1),(6)\},\]
    and
    \[|(0,1,1)|_2 = 2,\ |(2,0,1)|_2 = 3,\ |(0,3)|_2 = 3,\ |(2,2)|_2 = 4,\ |(4,1)|_2 = 5,\ |(6)|_2 = 6.\]
    so the theorem implies
    \[\sum_{j\geq 0}h^j_{st}(-6,6)\cdot t^j = t^2+2t^3+t^4+t^5+t^6.\]
    This example shows that the cohomology group $H^3\left(Fl_n,\mc{O}_{Fl_n}\left(-6,6,0,\cdots,0\right)\right)$ is decomposable for $n\gg 0$ (in fact for $n\geq 7$, see Theorem~\ref{thm:coh-lamn-poly}). By contrast, for $j=0,1$ and all $\ll$, the $\GL_n$-module $H^j\left(Fl_n,\mc{O}_{Fl_n}\left(\ll\right)\right)$ is always indecomposable \cite{andersen}*{Theorem~3.5}, \cite{jantzen}*{Corollaries~II.2.3,~II.5.16}.
    
    When $p=3$ we have
    \[ A_{3,6} = \{(3,1),(6)\},
    \quad\text{and}\quad
    |(3,1)|_3 = 3,\ |(6)|_3 = 4,
    \quad\text{thus}\quad
    \sum_{j\geq 0}h^j_{st}(-6,6)\cdot t^j = t^3+t^4.\]
    When $p=5$ we have
    \[ A_{5,6} = \{(1,1),(6)\},
    \quad\text{and}\quad
    |(1,1)|_5 = 2,\ |(6)|_5 = 3,
    \quad\text{thus}\quad
    \sum_{j\geq 0}h^j_{st}(-6,6)\cdot t^j = t^2+t^3.\]
\end{example}

Using \eqref{eq:stab-O(-d,d)} and the fact that stabilization of cohomology occurs for $n>d$, one can in fact see that the number of decomposable summands in $H^j(Fl_n,\mc{O}_{Fl_n}(\ll))$ cannot be bounded by a polynomial function of $j,n$ and the entries of $\ll$. Indeed, suppose for simplicity that $\chr(\kk)=2$ so that the total cohomology of $\mc{O}_{Fl_n}(-n+1,n-1,0,\cdots)$ is given precisely by the number of binary partitions of $n$, which has super-polynomial growth \cites{mahler,debruijn,froberg}. Since the number of cohomological degrees where cohomology is concentrated is bounded by (a polynomial in) $n$, the conclusion follows. The next example shows that the range of cohomological degrees where non-vanishing of cohomology occurs is usually disconnected (see also the discussion at \cite{hum-MOFW}).

\begin{example}\label{ex:disconnected-coh}
    For any prime $p$ we have
    \[A_{p,p^2} = \{(0,0,1),(0,p),(p^2-p,1),(p^2)\}\]
    and therefore the stable cohomology for $\ll=(-p^2,p^2)$ is given by
    \[\sum_{j\geq 0}h^j_{st}(-p^2,p^2)\cdot t^j = t+t^2+t^{2p-1}+t^{2p}.\]
    When $p\geq 3$, this yields two intervals when non-vanishing of cohomology occurs. The reader can check that for $\ll=(-p^3,p^3)$, $p\geq 3$, the cohomological degrees where $h^j_{st}(\ll)\neq 0$ form $p+2$ distinct intervals.
\end{example}

We can view \eqref{eq:stab-O(-d,d)} as the base case of a general recursion describing stable cohomology for weights of the form $\ll=(-a-b,a,1^b)$. To that end, we define the generating functions
\begin{equation}\label{eq:Hab}
 H_{a,b}(t) = \sum_{j\geq 0} h^j_{st}(-a-b,a,1^b)\cdot t^j,    
\end{equation}
along with
\begin{equation}\label{eq:def-Nbtu}
    \mc{N}_b(t,u) = \frac{u^b}{t^b}\cdot\sum_{a\geq 1} H_{a,b}(t)\cdot u^a.
\end{equation}
If we let
\begin{equation}\label{eq:def-Atu}
    \mc{A}(t,u) = \prod_{i\geq 1}\frac{1+t\cdot u^{p^i}}{1-t^2\cdot u^{p^i}} 
\end{equation}
then we can reformulate \eqref{eq:stab-O(-d,d)} as
\[ \mc{N}_0(t,u) =-1 + (1+t\cdot u) \cdot \mc{A}(t,u).\]
For more general values of the parameter $b$, we have the following simple recursion.

\begin{stable-recursion-hooks*}
If $\chr(\kk)=p>0$ then for $b\geq 0$
\begin{equation}\label{eq:Npb+p-1}
\mc{N}_{pb+p-1}(t,u) = \mc{N}_b(t,u^p),
\end{equation}
and for each $i=0,\cdots,p-2$,
\begin{equation}\label{eq:Npb+i}
\mc{N}_{pb+i}(t,u) = \mc{N}_b(t,u^p) + \left(t\cdot u^{pb+i+1} + t^2\cdot u^{pb+p}\right) \cdot \mc{A}(t,u).
\end{equation}
\end{stable-recursion-hooks*}

We prove the above identities in Section~\ref{subsec:generating}, based on more detailed formulas established in Theorem~\ref{thm:recursion-hooks}. It is then easy to derive a formula for stable cohomology when $b=p^r-1$: if we write $q=p^r$ then
\[ \sum_{a\geq 1} H_{a,q-1}(t)\cdot u^a = \frac{t^{q-1}}{u^{q-1}}\cdot\left[-1+(1+t\cdot u^q)\cdot\mc{A}(t,u^q)\right].\]
It also allows one to characterize the non-vanishing of $H_{a,b}(t)$: if we write $b=-1+c\cdot q$, with $p\nmid c$, then
\begin{equation}\label{eq:Hab-nonzero} H_{a,b}(t) \neq 0 \Longleftrightarrow pq|a-1\quad\text{ or }\quad pq|a+b.
\end{equation}

\begin{remark}\label{rem:strong-linkage}
    If we assume $n\gg 0$ and write $\ll$ instead of $\ll^{[n]}=(-a-b,a,1^b,0^{n-b-2})$, then it follows from the Strong Linkage Principle \cite{And-str-link} that if $\mc{O}(\ll)$ has non-vanishing (stable) cohomology then the zero weight $\ul{0}=(0^n)$ is strongly linked to $\chi=(a-1,0^b,-1^{a-1},0^{n-a-b})$. If $b=0$ and $a=d$, this forces $d\equiv 0,1\ (\op{mod}\ p)$, so \eqref{eq:stab-O(-d,d)} shows that the Strong Linkage Principle characterizes non-vanishing for weights $(-d,d,0,\cdots)$. When $b>0$ however, the condition $p|a+b$ suffices for $\ul{0}$ to be strongly linked to $\chi$ (or even very strongly linked, in the sense of \cite{wong}), but the characterization \eqref{eq:Hab-nonzero} shows that this is not enough to guarantee that the trivial representation $\kk$ appears in the cohomology of $\mc{O}(\ll)$.
\end{remark}

There are natural variations of the stability problem where we replace line bundles by higher rank vector bundles, and consider partial flag varieties instead of $Fl_n$. We illustrate this perspective in the case of projective space $\PP^{n-1}=\PP(\kk^n)$ and vector bundles $\mc{P}(\Omega)$, where $\Omega$ is the cotangent sheaf and $\mc{P}$ is a strict polynomial functor. The foundational work of Friedlander and Suslin \cite{fri-sus} introduced the theory of strict polynomial functors and established its homological framework, thereby laying the groundwork for much of the subsequent research on Ext groups, functor homology, and cohomological invariants for polynomial representations. The present work aims to establish initial connections between the theory of polynomial functors and sheaf cohomology, with the main applications occurring on the cohomology side. In subsequent work \cite{RV} we further develop these ideas by focusing on applications to extension groups between polynomial functors. 

Basic examples of polynomial functors include the symmetric powers $\Sym^d$, the exterior powers $\bw^r$ (in which case $\bw^r\Omega=\Omega^r$ is the sheaf of differential $r$-forms), or the divided powers $D^e$. If $\chr(\kk)=p>0$ then we have the \defi{Frobenius power functor} $F^p$ which is the unique simple subfunctor of $\Sym^p$. More generally, we write $\bb{L}_{\ll}$ for the unique simple subfunctor of $\bb{S}_{\ll}$. Effective versions of the results summarized by the following theorem are proved in Section~\ref{sec:poly-omega} (see Theorems~\ref{thm:coh-omega},~\ref{thm:stab-coh-omega},~\ref{thm:Kunneth-coh-omega},~\ref{thm:invariance-Frob}).

\begin{stable-proj*}
 If $\mc{P}$ is a polynomial functor then $H^j\left(\PP^{n-1},\mc{P}(\Omega)\right)$ has a trivial $\GL_n$-action and is independent of $n$ for $n\gg 0$. We write $H^j_{st}\left(\mc{P}(\Omega)\right)$ for the resulting \defi{stable cohomology group}, and denote by $h^j_{st}\left(\mc{P}(\Omega)\right)$ its dimension. We have:
 \begin{itemize}
     \item {\bf K\"unneth formula:} if $\mc{P}',\mc{P}''$ are polynomial functors, then
     \[H^j_{st}\left(\mc{P}'(\Omega)\oo\mc{P}''(\Omega)\right) = \bigoplus_{j'+j''=j} H^{j'}_{st}\left(\mc{P}'(\Omega)\right) \oo H^{j''}_{st}\left(\mc{P}''(\Omega)\right).\]
     \item {\bf Invariance under Frobenius:} if $\mc{P}$ is a polynomial functor then
     \[H^j_{st}\left((F^p\circ\mc{P})(\Omega)\right) = H^j_{st}\left((\mc{P}\circ F^p)(\Omega)\right) = H^j_{st}\left(\mc{P}(\Omega)\right).\]
 \end{itemize}
\end{stable-proj*}

Standard manipulations involving the projection formula \cite{BCRV}*{Sections~9.3,~9.8} imply that
\begin{equation}\label{eq:Flag-to-Pspace}
H^j_{st}(\bb{S}_{\ll}\Omega) = H^j_{st}(-|\ll|,\ll_1,\cdots,\ll_r)\quad\text{for every partition }\ll=(\ll_1,\cdots,\ll_r),
\end{equation}
where on the left side we have stable cohomology on projective space, while on the right side we have stable cohomology on the flag variety. In particular \eqref{eq:stab-O(-d,d)} describes the stable cohomology of $\Sym^d\Omega$, while the polynomial $H_{a,b}(t)$ in \eqref{eq:Hab} encodes the stable cohomology of $\bb{S}_{(a,1^b)}\Omega$ (we call $(a,1^b)$ a \defi{hook partition}). 

Using the K\"unneth formula and invariance under Frobenius in conjunction with Steinberg's tensor product theorem \cite{jantzen}*{Corollary~II.3.17}, it follows that in order to compute the stable cohomology of $\bb{L}_{\ll}\Omega$, it suffices to consider only the case when $\ll$ is \defi{$p$-restricted} (or \defi{column $p$-regular}):
\[ 0\leq \ll_i-\ll_{i+1} < p \text{ for all }i.\]
This seems to be difficult in general, but we can do it in the case when $\bb{L}_{\ll}=T_p\Sym^d$ is a \defi{truncated symmetric power}, as follows. We recall first that
\begin{equation}\label{eq:def-TpSymd} 
T_p\Sym^d = \coker\left(F^p\oo \Sym^{d-p} \lra \Sym^d\right)
\end{equation}
where we view $F^p$ as a subfunctor of $\Sym^p$ and the map is the usual polynomial multiplication. We also note that the partition $\ll$ we get in this case is given by $\ll=((p-1)^e,r)$ where $d=(p-1)e+r$ and $0\leq r<p-1$ is the remainder of the division of $d$ by $p-1$ \cite{doty-walker}*{Remark~0.2}. We show the following in Section~\ref{subsec:trunc-pows-Doty}.

\begin{stable-truncated-pows*}
 If $d\neq 0,1\ (\op{mod }p)$ then $h^j_{st}(T_p\Sym^d\Omega)=0$ for all $j$. Otherwise, using the notation \eqref{eq:def-mp-index}, we have
 \begin{equation}\label{eq:stable-coh-trunc-pows} h^{|d|_p}_{st}(T_p\Sym^d\Omega)=1\quad\text{ and }\quad h^j_{st}(T_p\Sym^d\Omega)=0\text{ for }j\neq |d|_p.
 \end{equation}
\end{stable-truncated-pows*}

This result is familiar in the case when $p=2$, when we have $T_p\Sym^d=\bw^d$ is an exterior power, hence $T_p\Sym^d\Omega=\Omega^d$ whose (stable) cohomology is one-dimensional, concentrated in degree $d=|d|_2$. We can now shed more light on the formula \eqref{eq:stab-O(-d,d)}, using the filtration on $\Sym^d$ constructed by Doty in \cite{doty1989submodules}: the composition factors are all of the form 
\[\mc{F}^{\ul{a}} = T_p\Sym^{a_0} \oo F^p\left(T_p\Sym^{a_1}\right) \oo F^{p^2}\left(T_p\Sym^{a_2}\right) \oo \cdots,\quad\text{where}\quad d = a_0+a_1p+a_2p^2+\cdots.\]
It follows from \eqref{eq:stable-coh-trunc-pows} (along with the K\"unneth formula and invariance under Frobenius) that $\mc{F}^{\ul{a}}(\Omega)$ has vanishing stable cohomology if $\ul{a}\not\in A_{p,d}$, and has $1$-dimensional cohomology concentrated in degree $|\ul{a}|_p$ if $\ul{a}\in A_{p,d}$. In particular, each of the simple composition factors in the Doty filtration contributes at most one dimension to the cohomology of $\Sym^d\Omega$, and the spectral sequence associated to the Doty filtration of $\Sym^d\Omega$ degenerates. We note that each $\mc{F}^{\ul{a}}$ is a simple polynomial functor, and that in general there are examples of such functors $\bb{L}_{\mu}$ for which $\bb{L}_{\mu}(\Omega)$ has more than one non-vanishing cohomology (see Example~\ref{ex:simpleWithMoreCohom}).

Many of the examples we have seen so far of polynomial functors $\mc{P}$ for which the stable cohomology of~$\mc{P}(\Omega)$ vanishes identically can be explained by the next theorem, which is subsumed by Theorems~\ref{thm:Weyl-Omega} and~\ref{thm:vanishing-p-core}.

\begin{stable-vanishing*}
 The following classes of polynomial functors~$\mc{P}$ satisfy $h^j_{st}(\mc{P}(\Omega))=0$ for all $j$:
 \begin{itemize}
     \item For any $\kk$, the Weyl functors $\mc{P}=\bb{W}_{\ll}$ where $\ll$ is a partition with $\ll_1\geq 2$.
     \item If $\chr(\kk)=p$, the functors $\mc{P}$ admitting a filtration with composition factors $\bb{L}_{\ll}$ where $\ll$ is a partition whose $p$-core $\mu$ satisfies $\mu_1\geq 2$. 
 \end{itemize}
\end{stable-vanishing*}

The main examples of functors satisfying the last condition in the theorem are the Schur functors $\bb{S}_{\ll}$ where $\mu=\op{core}_p(\ll)$ has $\mu_1\geq 2$. In particular, if $\ll=(d)$ and $d=ap+b$ with $0\leq b<p$, then $\mu = (b)$: this implies that $\bb{S}_{\ll}\Omega=\Sym^d\Omega$ has vanishing cohomology if $d\not\equiv 0,1\text{ (mod $p$)}$, which was noted before. Similarly, for the truncated Schur functors $T_p\Sym^d = \bb{L}_{\ll}$, where $\ll = ((p-1)^e,r)$, we have the same core $\mu=(b)$, and we get vanishing as noted when $b\neq 0,1$, that is, $d\not\equiv 0,1\text{ (mod $p$)}$. It follows from \eqref{eq:Hab-nonzero} that if $b\equiv -1\text{ (mod $p$)}$ and $a\not\equiv 1\text{ (mod $p$)}$ then $H_{a,b}(t)=0$: this can be seen alternatively by noting that hypotheses imply that the $p$-core of $\ll=(a,1^b)$ is $\mu=(a',1^{p-1})$ with $2\leq a'\leq p$ and $a'\equiv a\text{ (mod $p$)}$. Nevertheless, we do not know of a general argument that explains the full strength of \eqref{eq:Hab-nonzero}.

There is one more general class of partitions $\ll$ for which we can determine the stable cohomology of $\bb{S}_{\ll}\Omega$, namely the two-column partitions ($\ll_1=2$). We do so by uncovering a surprising relationship with the stable cohomology for hook partitions. Most remarkably, this relationship does not depend on the characteristic of the field $\kk$. Specifically, we prove the following in Theorem~\ref{thm:dualityFor2ColAndHooks}.

\begin{stable-twocolumn-hooks*}
 Suppose that $\ll=(2^d,1^{m-d})$ is the partition whose transpose is $(m,d)$. We have that
 \[ h^i_{st} (\bbs_{\lambda} (\Omega) ) = h^{2m+1 - i}_{st} (\bbs_{(d+1,1^{m-d})} (\Omega)) \quad \text{for all} \ i.
 \]
 Equivalently, on the flag variety we have
 \[ h^i_{st} (-m-d,2^d,1^{m-d}) = h^{2m+1 - i}_{st} (-m-1,d+1,1^{m-d}) \quad \text{for all} \ i.
 \]
\end{stable-twocolumn-hooks*}

Our approach to studying the stable cohomology of $\bb{S}_{\ll}\Omega$ is based on the construction of explicit universal complexes (over~$\bb{Z}$) whose homology, after base change to the field $\kk$, computes $H^j_{st}(\bb{S}_{\ll}\Omega)$ over $\kk$. These complexes are obtained by resolving $\bb{S}_{\ll}\Omega$ by tensor products of exterior powers of $\Omega$ and taking hypercohomology, as explained in Section~\ref{sec:resolutions-ext-pows}. The study of such resolutions was pioneered by Akin and Buchsbaum in the 80s \cites{AB1,AB2} who notably showed their existence for all (skew-)partitions $\ll$, and later on the theory was developed by Buchsbaum and Rota \cites{BR1,BR2,BR3}, and others. The most explicit general construction of such resolutions was obtained by Santana and Yudin \cite{san-yud}, but the said resolutions are
longer than the global dimension of the relevant category of representations \cite{totaro}*{Theorem~2}. We hope that this work will bring renewed interest in the search for an optimal resolution of $\bb{S}_{\ll}$, particularly one that would be suitable for such explicit calculations of stable cohomology as in the case of hooks and $2$-column partitions.

Besides it being such a fundamental problem, our interest in computing cohomology on (partial) flag varieties comes from its utility in describing many homological invariants in commutative algebra. Most famously the syzygies of generic determinantal ideals were obtained using such calculations in characteristic zero by Lascoux \cite{lascoux}, and later these geometric methods have played a major role in the work of Weyman (one can find many applications in his book \cite{weyman2003cohomology}). Many of the results are however restricted to characteristic zero, because they are based on the Borel--Weil--Bott theorem, although in general they do not require the full generality of this theorem. One can often get interesting applications by focusing on the cohomology of special line (or vector) bundles. We conclude below with an application of our results for $\Sym^d\Omega$ (or the line bundles associated to the weights $(-d,d,0,\cdots)$) to a sharp vanishing theorem for Koszul modules over arbitrary characteristics (see Theorem~\ref{thm:vanishing-koszul}, and Section~\ref{sec:vanish-Koszul} for more context).

Let $V=\kk^n$ be a vector space of dimension $n\geq 3$ and let $K\subseteq\bw^2 V$ be a subspace of dimension $m$. We let $S=\Sym(V)$ denote the symmetric algebra on $V$, and consider the $3$-term complex of free graded $S$-modules
\begin{equation}\label{eqn:W}
\xymatrixcolsep{5pc}
\xymatrix{
K \oo S \ar[r]^{\delta_2|_{K \oo S}} & V\oo S(1) \ar[r]^{\delta_1} & S(2)
}
\end{equation}
where $S(j)$ denotes the free graded rank one module, with generator in degree $-j$: the grading is such that $S(j)_d = S_{d+j}$. In particular, the leftmost module is generated in degree $0$. Moreover, $\delta_1:V\oo S\lra S$ is the natural multiplication map, and 
\[\delta_2:\bw^2 V \oo S \lra V \oo S,\quad (v\wedge v') \oo f \overset{\delta_2}{\lra} v\oo (v' f) - v' \oo (v f)\mbox{ for }v,v'\in V, f\in S,\] 
is the second differential of the Koszul complex on $V$. The \defi{Koszul module} $W(V,K)$ is defined to be the middle homology of (\ref{eqn:W}), and it is a graded module generated in degree $0$. In Section~\ref{sec:vanish-Koszul} we prove:

\begin{vanishing-Koszul*}
 We have the equivalence
 \[W_j(V,K)=0\quad\text{for }j\gg 0 \quad\Longleftrightarrow\quad W_j(V,K)=0\quad\text{for }j\geq 2n-7.\]
 Moreover, if $m=2n-3$, $p=\chr(\kk)>0$, and the vanishing above holds, then
 \[ W_{2n-8}(V,K) \neq 0 \quad\Longleftrightarrow\quad W_{2n-8}(V,K)\simeq\kk \quad \quad\Longleftrightarrow\quad n = 3 + p^k\text{ for some }k. \]
\end{vanishing-Koszul*}

If $\op{char}(\kk)=0$ or $p\geq n-2$, it was shown in \cite{AFPRW} that the vanishing of $W_j(V,K)$ holds for $j\geq n-3$ and that this is a sharp estimate. The result constituted the main ingredient of a new proof of Green's conjecture for generic curves, and its natural extension to positive characteristic (see also \cite{rai-sam}). It was known by work of Schreyer \cite{schreyer} that Green's conjecture for generic curves can fail in small characteristic, and our theorem gives an algebraic counterpart explaining the failure of vanishing for all finite length Koszul modules in small characteristics.

\bigskip

\noindent{\bf Organization.} In Section~\ref{sec:prelim} we recall basic aspects of the theory of polynomial functors, along with the terminology on partitions, cores, and the consequences of Strong Linkage that will be relevant to our work. In Section~\ref{sec:stab-flag-coh} we discuss the stabilization theorem for line bundles on flag varieties, while in Section~\ref{sec:poly-omega} we discuss the main stabilization results for polynomial functors of the cotangent sheaf on projective space. Section~\ref{sec:resolutions-ext-pows} discusses the Akin--Buchsbaum resolutions by tensor products of exterior powers, and constructs the arithmetic complexes encoding the stable cohomology for several classes of polynomial functors. Section~\ref{sec:arithmetic} contains the main explicit calculations of stable cohomology, and explains the duality between hooks and two-column partitions. Finally, Section~\ref{sec:vanish-Koszul} describes the vanishing theorem for finite length Koszul modules.

\section{Preliminaries}\label{sec:prelim}

Throughout the paper, $\kk$ denotes an algebraically closed field of characteristic $\chr(\kk)=p$ (usually $p>0$). 

\subsection{Quick review of polynomial functors}\label{subsec:poly-fun}

Let $\mc{V}$ denote the category of finite dimensional $\kk$-vector spaces. A functor $\mc{P}:\mc{V}\lra\mc{V}$ is a \defi{polynomial functor} if the natural map $\Hom_{\kk}(V,W) \lra \Hom_{\kk}(\mc{P}(V),\mc{P}(W))$ is polynomial for $V,W\in\mc{V}$. We refer the reader to \cite{fri-sus} where a theory of \defi{strict polynomial functors} is developed over an arbitrary field $\kk$. As remarked after \cite{fri-sus}*{Definition 2.1}, the definition we use here is sufficient in our context where $\kk$ is algebraically closed. The polynomial functor $\mc{P}$ is \defi{homogeneous of degree~$d$} if $\mc{P}(c\cdot\op{id}_V) = c^d\cdot\op{id}_V$ for $c\in\kk$ and $V\in\mc{V}$, where $\op{id}_V$ denotes the identity map on $V$. 

We will be interested exclusively in homogeneous polynomial functors of degree $d$, and we write $\mf{Pol}_d$ for the corresponding category. There is an equivalence of categories between $\mf{Pol}_d$ and its opposite category $\mf{Pol}_d^{op}$, sending $\mc{P}$ to the functor $\mc{P}^{\#}$ defined by $\mc{P}^{\#}(V)=\left(\mc{P}(V^{\vee})\right)^{\vee}$ \cite{fri-sus}*{Proposition~2.6}. The main examples of objects in $\mf{Pol}_d$ are the \defi{Schur functors} $\bb{S}_{\ll}$ for $\ll$ a partition of $d$, and the \defi{Weyl functors} $\bb{W}_{\ll}=\bb{S}_{\ll}^{\#}$. When $\ll=(d)$ we have $\bb{S}_{\ll}=\Sym^d$ is the $d$-th symmetric power functor, while $\bb{W}_{\ll}=\op{D}^d$ is the $d$-th divided power functor. When $\ll=(1^d)$, both $\bb{S}_{\ll}$ and $\bb{W}_{\ll}$ agree with the exterior power functor $\bw^d$. There is an equivalence of $\mf{Pol}_d$ with the category of polynomial representations of degree $d$ of the algebraic group $\GL_n(\kk)$ whenever $n\geq d$, which in turn is equivalent to the category of finitely generated modules over the Schur algebra $S(n,d)$ (see \cite{fri-sus}*{Section~3}, and also \cite{green}, \cite{jantzen} for more about Schur algebras and representations of algebraic groups). We write $\bb{L}_{\ll}$ for the unique simple subfunctor of $\bb{S}_{\ll}$, which is also the unique simple quotient of~$\bb{W}_{\ll}$, and satisfies $\bb{L}_{\ll}=\bb{L}_{\ll}^{\#}$. When $\ll=(q)$ where $q=p^k$ is a power of $p=\chr(\kk)$, we have that $\bb{L}_{\ll}=F^q$ is the \defi{Frobenius $q$-power functor}: it is the subfunctor of $\Sym^q$ given for $V\in\mc{V}$ by
\[F^qV = \{v^q : v\in V\} \subset \Sym^q V.\]

\subsection{Young diagrams, cores, and strong linkage}

Given a partition $\ll=(\ll_1\geq\ll_2\geq\cdots)$, its Young diagram consists of left-justified rows of boxes, with $\ll_i$ boxes in row $i$. The \defi{transpose partition $\ll'$} is the partition corresponding to the transposed Young diagram of $\ll$. We usually identify $\ll$ with its Young diagram, which explains the following terminology. A partition $\lambda$ of the form $\lambda = (a,1^b)$ for integers $a,b \geq 0$ is called a \defi{hook partition}. A partition $\lambda$ with each part $\ll_i\leq 2$ is said to be a \defi{two-column partition}.
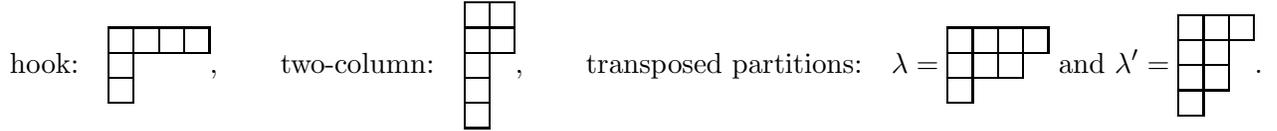
\[
\ytableausetup{centertableaux,smalltableaux}
\text{hook:}\quad\ydiagram{4,1,1},\qquad
\text{two-column:}\quad\ydiagram{2,2,1,1,1},\qquad
\text{transposed partitions:}\quad\ll=\ydiagram{4,3,1}\text{ and }\ll'=\ydiagram{3,2,2,1}.
\]
We write $\mu\subseteq\ll$ if $\mu_i\leq\ll_i$ for all $i$ (equivalently, the diagram of $\mu$ is contained in that of $\ll$; in Section~\ref{subsec:ribbons} we will consider generalizations of partitions and their Young diagrams given by certain skew-shapes $\ll/\mu$ for $\mu\subset\ll$). If $|\ll|=|\mu|$ then we say that $\mu\leq\ll$ in the \defi{dominance order} if $\mu_1+\cdots+\mu_i\leq \ll_1+\cdots+\ll_i$ for all $i$. For the pair of transposed partitions above, we have $\ll'\leq\ll$.

We index the boxes of a partition $\ll$ by pairs $(i,j)$, $1\leq j\leq \ll_i$, and call the \defi{rim} of $\ll$ the collection of boxes $(i,j)\in\ll$ with $(i+1,j+1)\not\in\ll$: we use $\bullet$ to indicate the rim boxes in the following diagram
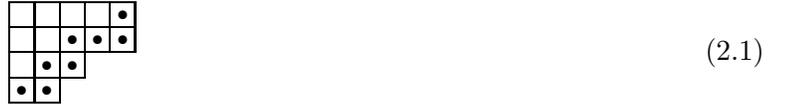
\begin{equation}\label{eq:5532-rim}
\ytableausetup{notabloids}
\begin{ytableau}
\ & \ & \ & \ & \bullet \\
\  & \ & \bullet & \bullet & \bullet \\
\ & \bullet & \bullet \\
\bullet & \bullet
\end{ytableau}
\end{equation}
Given $p>0$, a \defi{rim $p$-hook} of $\ll$ is a collection of $p$ consecutive boxes in the rim of $\ll$, whose removal results in another partition $\mu$. The \defi{$p$-core of $\ll$} is the partition $\mu=\op{core}_p(\ll)$ obtained from $\ll$ by successively removing rim $p$-hooks while they exist. For $\ll=(5,5,3,2)$ as in \eqref{eq:5532-rim} and $p=7$, we have $\op{core}_p(\ll)=(1)$, which can be obtained as follows:
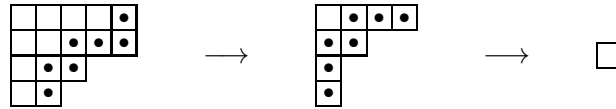
\[
\begin{ytableau}
\ & \ & \ & \ & \bullet \\
\  & \ & \bullet & \bullet & \bullet \\
\ & \bullet & \bullet \\
\ & \bullet
\end{ytableau}
\qquad\lra\qquad
\begin{ytableau}
\ & \bullet & \bullet & \bullet  \\
\bullet  & \bullet \\
\bullet \\
\bullet 
\end{ytableau}
\qquad\lra\qquad
\ydiagram{1}
\]
It is a non-trivial fact that $\op{core}_p(\ll)$ is independent of the order in which rim hooks are removed (see for instance \cite{james1978some}): the reader can check that for $\ll=(5,5,3,2)$ there is one additional rim $7$-hook removal sequence resulting in the partition $(1)$.

It is a consequence of the Strong Linkage Principle that if $\bb{L}_{\mu}$ is a simple composition factor for $\bb{S}_{\ll}/\bb{L}_{\ll}$ (or equivalently for $\ker(\bb{W}_{\ll}\onto\bb{L}_{\ll})$) then $\mu < \ll$ in the dominance order, and $\op{core}_p(\mu)=\op{core}_p(\ll)$. An instance of this appears in Example~\ref{ex:simpleWithMoreCohom} where $p=2$, $\ll=(2,2,1,1)$ and the composition factors of $\bb{S}_{\ll}$ are $\bb{L}_{\ll}$ (with multiplicity one), $\bb{L}_{(2,1,1,1,1)}$ (with multiplicity one), and $\bb{L}_{(1,1,1,1,1,1)}=\bw^6$ (with multiplicity two). In particular, if $\ll$ is minimal in the dominance order among all partitions of $d$ with a given $p$-core (such as $\ll=(1^d)$), then $\bb{S}_{\ll}=\bb{L}_{\ll}=\bb{W}_{\ll}$.

In Section~\ref{sec:stab-flag-coh} we will need to consider not necessarily polynomial representations of $\GL_n$. The simple ones are classified by \defi{dominant weights} $\mu\in\bb{Z}^n_{dom}$, which are elements $\mu\in\bb{Z}^n$ with $\mu_1\geq\cdots\geq\mu_n$. We write $L(\mu)$ for the simple of highest weight $\mu$, and note that when $\mu$ is a partition ($\mu_n\geq 0$) then $L(\mu)=\bb{L}_{\mu}(\kk^n)$. We write $\rho=(n-1,n-2,\cdots,0)\in\bb{Z}^n$ and recall the usual action of the symmetric group $\mf{S}_n$ on $\bb{Z}^n$ by coordinate permutations, as well as the \defi{$\bullet$-action} defined by
\[\sigma\bullet\ll = \sigma(\ll+\rho)-\rho.\]

\section{Polynomial stabilization for cohomology of line bundles}\label{sec:stab-flag-coh}

The goal of this section is to provide an effective stabilization result for the cohomology of line bundles $\mc{O}_{Fl_n}(\ll_1,\cdots,\ll_r,0,\cdots,0)$, as $n\lra\infty$. Recalling the notation \eqref{eq:ll-padded}, we start by proving the following.

\begin{theorem}\label{thm:coh-lamn-poly}
 Suppose that $\ll\in\bb{Z}^r$ and
 \[ n\geq i - \ll_i \text{ for }i=1,\cdots,r.\]
 We have that $H^j\left(Fl_n,\mc{O}_{Fl_n}(\ll^{[n]})\right)$ is a polynomial $\GL_n$-representation for all $j\geq 0$.
\end{theorem}

\begin{proof} Suppose that $L(\mu)$ is a simple composition factor of $H^j\left(Fl_n,\mc{O}_{Fl_n}(\ll^{[n]})\right)$, where $\mu\in\bb{Z}^n_{dom}$. By \cite{fri-sus}*{Corollary~3.5.1}, it suffices to prove that $\mu_n\geq 0$. Suppose by contradiction that $\mu_n<0$ and let $\chi\in\bb{Z}^n_{dom}-\rho$ and $w\in\mf{S}_n$ such that $\ll^{[n]} = w\bullet\chi$, so that the sequences
\[(\ll_1-1,\cdots,\ll_r-r,-r-1,\cdots,-n)\quad\text{and}\quad(\chi_1-1,\cdots,\chi_n-n)\]
agree up to a permutation of the entries. By the Strong Linkage Principle \cite{And-str-link}*{Theorem~1}, we have that $\mu$ is strongly linked to $\chi$, which implies that $\chi_n\leq\mu_n\leq -1$. It follows that $\chi_n-n\leq -n-1$ and therefore we must have 
\[\ll_i-i = \chi_n-n \leq -n-1\quad\text{ for some }i=1,\cdots,r.\]
This forces $n=i-\ll_i+\chi_n<i-\ll_i$, which contradicts our hypothesis.
\end{proof}

Given a $\GL_n$-representation $M$, we write $M_{\ul{a}}$ for the weight space of $M$ corresponding to the weight $\ul{a}\in\bb{Z}^n$.

\begin{corollary}\label{cor:dega-stab-coh=0}
 With the notation in Theorem~\ref{thm:coh-lamn-poly}, if $\ul{a}\in\bb{Z}^n$ satisfies $a_i<0$ for some $i$ then 
 \[H^j\left(Fl_n,\mc{O}_{Fl_n}(\ll^{[n]})\right)_{\ul{a}}=0.\]
\end{corollary}

\begin{proof}
 This follows from Theorem~\ref{thm:coh-lamn-poly} and \cite{fri-sus}*{Proposition~3.5}.
\end{proof}

\begin{theorem}\label{thm:iso-coh-dega}
Suppose that $\ll\in\bb{Z}^r$ and $n\geq i-\ll_i$ for $1\leq i\leq r$. Let $\mc{Q}$ denote the tautological rank $n$ quotient sheaf on $Fl_{n+1}$, and write $\iota:Fl_n\lra Fl_{n+1}$ for the inclusion coming from the identification of $Fl_n$ with the vanishing locus $Z(x_{n+1})$ of the section $x_{n+1}\in H^0(Fl_{n+1},\mc{Q})$ corresponding to the last coordinate function on $\kk^{n+1}$. The natural map
     \[ H^j\left(Fl_{n+1},\mc{O}_{Fl_{n+1}}(\ll^{[n+1]})\right)_{(\ul{a},0)} \lra H^j\left(Fl_{n},\mc{O}_{Fl_{n}}(\ll^{[n]})\right)_{\ul{a}}\]
     is an isomorphism for all $j\geq 0$ and all $\ul{a}\in\bb{Z}^n$.
\end{theorem}

\begin{proof}
 We write $G=\GL_n \times\, \kk^{\times}\subset\GL_{n+1}$, and write $x_{n+1}^i\cdot\kk$ for the $1$-dimensional $G$-module where $\GL_n$ acts trivially, and $\kk^{\times}$ acts on $x_{n+1}$ via the usual scaling action. For a $G$-equivariant sheaf $\mc{F}$, we write $x_{n+1}^i\cdot\mc{F} = x_{n+1}^i\cdot\kk \oo_{\kk} \mc{F}$. The section $x_{n+1}$ gives rise to a $G$-equivariant map $x_{n+1}\cdot\mc{Q}^{\vee} \lra \mc{O}_{Fl_{n+1}}$ whose image is the ideal sheaf of $Fl_n$, and hence is resolved by the associated Koszul resolution. Using the fact that $\iota^*\left(\mc{O}_{Fl_{n+1}}(\ll^{[n+1]})\right) = \mc{O}_{Fl_{n}}(\ll^{[n]})$, we get a $G$-equivariant resolution of $\mc{O}_{Fl_n}(\ll^{[n]}) = \iota_*\left(\mc{O}_{Fl_n}(\ll^{[n]})\right)$ by locally free sheaves $x_{n+1}^i \cdot\mc{K}_i$ on $Fl_{n+1}$, where
\[ \mc{K}_i = \bw^i\mc{Q}^{\vee} \oo \mc{O}_{Fl_{n+1}}(\ll^{[n+1]}) \quad\text{ for }i=0,\cdots,n.\]
The desired conclusion follows by considering the associated hypercohomology spectral sequence, if we can show that $H^j\left(Fl_{n+1},x_{n+1}^i \cdot\mc{K}_i\right)_{(\ul{a},0)}=0$ for $i>0$ and $\ul{a}\in\bb{Z}^n$. This in turn follows once we verify that
\begin{equation}\label{eq:HK_i=0-neg-weights} H^j\left(Fl_{n+1},\mc{K}_i\right)_{(\ul{a},b)}=0\quad\text{ for }b<0,\ \ul{a}\in\bb{Z}^n,
\end{equation}
which we do next. We write $\varepsilon_r$ for the $r$-th standard unit vector, and
\[ \varepsilon_R = \varepsilon_{r_1} + \cdots + \varepsilon_{r_i} \quad\text{ for }R = \{r_1,\cdots,r_i\}.\]
The sheaf $\mc{K}_i$ has a natural filtration (see for instance \cite{BCRV}*{Lemma~9.4.2}) with composition factors given by line bundles of the form 
\[ \mc{L}_R = \mc{O}_{Fl_{n+1}}(\ll^{[n+1]}-\varepsilon_R)\quad\text{ with }R = \{r_1<\cdots<r_i\}\subseteq\{1,\cdots,n\},\]
so it suffices to prove the vanishing in \eqref{eq:HK_i=0-neg-weights} for each $\mc{L}_R$. If $r_i>r$ then the weight $\mu=\ll^{[n+1]}-\varepsilon_R$ satisfies $\mu_{r_i}-\mu_{r_i+1}=-1$, and therefore $H^j\left(Fl_{n+1},\mc{L}_R\right)=0$ for all $j$ (for instance by \cite{BCRV}*{Lemma~9.3.4}). We may therefore assume that $r_i\leq r$ and thus the weight $\mu=\ll^{[n+1]}-\varepsilon_R$ satisfies
\[ \ll_i-1 \leq \mu_i \leq \ll_i \quad\text{for }i=1,\cdots,r,\quad \mu_i=0\quad\text{for }i>r.\]
It follows that $n+1\geq i-\mu_i$ for $i=1,\cdots,r$ and we can apply Corollary~\ref{cor:dega-stab-coh=0} to deduce that 
\[ H^j\left(Fl_{n+1},\mc{L}_R\right)_{(\ul{a},b)}=0\quad\text{when }b<0,\] 
concluding our proof.
\end{proof}

\begin{corollary}\label{cor:stab-coh=poly-functor}
 For $\ll\in\bb{Z}^r$ and $j\geq 0$, there exists a polynomial functor $\mc{P}_{\ll}^j$ of degree $|\ll|$ such that
 \[ H^j\left(Fl_{n},\mc{O}_{Fl_{n}}(\ll^{[n]})\right) = \mc{P}_{\ll}^j(\kk^n) \text{ for all }n\geq\max\{i-\ll_i:i=1,\cdots,r\}.\]
\end{corollary}

\begin{proof}
    Using the equivalence between polynomial functors of degree $|\ll|$ and polynomial representations of $\GL_n$ \cite{fri-sus}*{Lemma~3.4}, together with Theorem~\ref{thm:iso-coh-dega} and \cite{green}*{(6.5c)},  we obtain the desired conclusion. 
\end{proof}

\section{Polynomial functors of $\Omega$}\label{sec:poly-omega}

The goal of this section is to study the cohomology of $\mc{P}(\Omega)$ when $\mc{P}$ is a polynomial functor and $\Omega$ is the cotangent sheaf on projective space $\PP^{n-1}$. Although for small values of $n$ this can be quite complicated, we prove that as $n\to\infty$ the cohomology stabilizes, and we provide effective bounds for the stable range. In later sections we will give a number of explicit calculations of stable cohomology, but for now we focus on more qualitative statements. The most familiar instance of stabilization occurs when $\mc{P}=\bw^d$ is an exterior power functor. In that case $\mc{P}(\Omega) = \Omega^d$ is the sheaf of $d$-differential forms when $n\geq d+1$ (or zero when $n\leq d$), in which case the only non-zero cohomology group is given by
\begin{equation}\label{eq:Hd-of-Omegad}
H^d\left(\PP^{n-1},\Omega^d\right) = \kk.
\end{equation}
For more general polynomial functors we prove the following.

\begin{theorem}\label{thm:coh-omega}
 Suppose that $\mc{P}$ is a polynomial functor of degree $d$.
 \begin{enumerate}
  \item The action of $\GL_n$ on $H^j\left(\PP^{n-1},\mc{P}(\Omega)\right)$ is trivial for all $j<n-1$.
  \item If $n-1\geq d$ then the action of $\GL_n$ on $H^{n-1}\left(\PP^{n-1},\mc{P}(\Omega)\right)$ is trivial.
  \item If $e<0$ then $H^j\left(\PP^{n-1},\mc{P}(\Omega)\oo\mc{O}_{\PP^{n-1}}(e)\right)=0$ for all $j<n-1$.
  \item If $d-n+1\leq e\leq -1$ then $H^{n-1}\left(\PP^{n-1},\mc{P}(\Omega)\oo\mc{O}_{\PP^{n-1}}(e)\right)=0$.
 \end{enumerate}
\end{theorem}

\begin{theorem}\label{thm:stab-coh-omega}
 Let $\mc{P}$ be a polynomial functor of degree $d$.
 \begin{enumerate}
     \item For a fixed $j$, the groups $H^j\left(\PP^{n-1},\mc{P}(\Omega)\right)$ are independent of $n$ as long as $j<n-1$.
     \item The groups $H^j\left(\PP^{n-1},\mc{P}(\Omega)\right)$ are independent of $n$ as long as $d\leq n-1$.
     \item We have $H^j\left(\PP^{n-1},\mc{P}(\Omega)\right)=0$ for $j>d$.
 \end{enumerate}
\end{theorem}

In light of Theorem~\ref{thm:stab-coh-omega}, we can define the \defi{stable cohomology of $\mc{P}(\Omega)$} to be
\[H^j_{st}(\mc{P}(\Omega)) = H^j\left(\PP^{n-1},\mc{P}(\Omega)\right) \text{ for any }n>j+1.\]
Likewise, the \defi{stable cohomology polynomial} $h_{\mc{P} (\Omega)} (t)$ is defined to be the polynomial
$$h_{\mc{P} (\Omega)} (t) := \sum_{j \geq 0} \dim_\kk H^j_{st} (\mc{P} (\Omega))\cdot t^j.$$

The following theorems provide effective versions of the K\"unneth formula and the invariance under Frobenius discussed in the Introduction.

\begin{theorem}\label{thm:Kunneth-coh-omega}
    Suppose that $\mc{P}$, $\mc{P}'$ are polynomial functors of degrees $d,d'$ respectively. If $n>d+d'$ then 
    \[ H^j\left(\PP^{n-1},\mc{P}(\Omega)\oo\mc{P}'(\Omega)\right)=\bigoplus_{j_1+j_2=j}H^{j_1}\left(\PP^{n-1},\mc{P}(\Omega)\right) \oo H^{j_2}\left(\PP^{n-1},\mc{P}'(\Omega)\right). \]
\end{theorem}

\begin{theorem}\label{thm:invariance-Frob}
    Suppose that $\mc{P}$ is a polynomial functor of degree $d$. If $n-1\geq dp$ then
     \[H^j\left(\PP^{n-1},(F^p\circ\mc{P})(\Omega)\right) = H^j\left(\PP^{n-1},(\mc{P}\circ F^p)(\Omega)\right) = H^j\left(\PP^{n-1},\mc{P}(\Omega)\right).\]
\end{theorem}

\subsection{Trivial $\GL_n$-action and vanishing for negative twists}

The goal of this section is to prove Theorem~\ref{thm:coh-omega}. We write $\PP=\PP^{n-1}$, $V=\kk^n=H^0\left(\PP,\mc{O}_{\PP}(1)\right)$ and consider the Euler sequence
\begin{equation}\label{eq:Euler-sequence} 
0 \lra \Omega \lra V \oo \mc{O}_\PP(-1) \overset{\phi}{\lra} \mc{O}_\PP \lra 0,
\end{equation}
which is the basis for many of the subsequent constructions. We note that in order to prove Theorem~\ref{thm:coh-omega} for all polynomial functors of degree $d$, it suffices to prove it for all simple polynomial functors $\mc{P}=\bb{L}_{\mu}$, where $\mu$ is a partition of $d$. Our arguments will then be based on a double induction: first on the degree $d$, and then on the partition $\mu$ according to the dominance order. The proof strategy goes as follows:
\begin{itemize}
    \item Assuming Theorem~\ref{thm:coh-omega} is true for polynomial functors of degree $<d$, we verify it for all Schur functors $\mc{P}=\bb{S}_{\mu}$, where $\mu$ is a partition of $d$.
    \item Once we know the results for Schur functors $\bb{S}_{\mu}$ with $|\mu|=d$, we prove the theorem for $\bb{L}_{\mu}$ by induction on $\mu$, and conclude that it holds for all polynomial functors $\mc{P}$ of degree $d$.
\end{itemize}

\begin{proof}[Proof of Theorem~\ref{thm:coh-omega}]
    If $d=0$ then $\mc{P}(\Omega)$ is a direct sum of copies of $\mc{O}_{\PP}$ and the results follow from \cite{hartshorne}*{Theorem~III.5.1}. We assume that $d>0$ and that the theorem is true for all polynomial functors of degree~$<d$. 
        
    We next consider a Schur functor $\bb{S}_{\mu}$, where $\mu$ is a partition of $d$. Using \eqref{eq:Euler-sequence} and \cite{kou}*{Theorem~1.4(b)} we get a short exact sequence
    \[ 0 \lra \bb{S}_{\mu}\Omega \lra \bb{S}_{\mu}V\oo \mc{O}_\PP(-d) \lra \mc{E} \lra 0\]
    where $\mc{E}$ has a filtration with composition factors of the form $\mc{P}(\Omega)$ where $\mc{P}$ is a polynomial functor (in fact, a skew Schur functor) of degree $<d$. It follows by induction on $d$ that $H^j(\PP,\mc{E})$ has trivial $\GL_n$-action for $j<n-1$, as well as for $j=n-1$ when $d\leq n-1$. Since $H^j\left(\PP,\bb{S}_{\mu}V\oo \mc{O}_\PP(-d)\right)=0$ for $j<n-1$ we get
    \begin{equation}\label{eq:iso-HjSmu-Hj-1E}
    H^j\left(\PP,\bb{S}_{\mu}\Omega\right) \simeq H^{j-1}\left(\PP,\mc{E}\right)\quad \text{ for }j<n-1,
    \end{equation}
    hence Theorem~\ref{thm:coh-omega}(1) holds for $\bb{S}_{\mu}$ by induction on $d$. If $d\leq n-1$ then $H^{n-1}\left(\PP,\bb{S}_{\mu}V\oo \mc{O}_\PP(-d)\right)=0$, hence \eqref{eq:iso-HjSmu-Hj-1E} also holds for $j=n-1$, proving Theorem~\ref{thm:coh-omega}(2) for $\bb{S}_{\mu}$. If $e<0$ then $H^j\left(\PP,\bb{S}_{\mu}V\oo \mc{O}_\PP(e-d)\right)=0$ for $j<n-1$, hence
    \begin{equation}\label{eq:vanishing-HjSmu-Hj-1E}
    H^j\left(\PP,\bb{S}_{\mu}\Omega\oo \mc{O}_\PP(e)\right) \simeq H^{j-1}\left(\PP,\mc{E}\oo \mc{O}_\PP(e)\right)=0\quad \text{ for }j<n-1,
    \end{equation}    
    where the vanishing of cohomology for $\mc{E}\oo \mc{O}_\PP(e)$ follows from that of its composition factors (which is true by induction on $d$), proving Theorem~\ref{thm:coh-omega}(3) for $\mc{P}=\bb{S}_{\mu}$. Finally, if $-n+1\leq e-d\leq -1$, then $H^{n-1}\left(\PP,\bb{S}_{\mu}V\oo \mc{O}_\PP(e-d)\right)=0$, hence \eqref{eq:vanishing-HjSmu-Hj-1E} also holds for $j=n-1$, proving Theorem~\ref{thm:coh-omega}(4) for $\bb{S}_{\mu}$.
        
    To prove Theorem~\ref{thm:coh-omega} for $\bb{L}_{\mu}\Omega$, we start with the smallest partition $\mu$ of $d$ in the dominance order, namely $\mu=(1^d)$, in which case $\bb{L}_{\mu}=\bw^d$ and $\bb{L}_{\mu}\Omega=\Omega^d$. Since $\Omega^d=0$ for $d\geq n$, we may further assume $d\leq n-1$. The conclusions (1)-(4) follow readily by induction from the exact sequence
    \[ 0 \lra \Omega^d \lra \bw^d V\oo \mc{O}_{\PP}(-d) \lra \Omega^{d-1} \lra 0.\]
    Suppose now that $\mu_1\geq 2$ and consider the short exact sequence
    \[ 0 \lra \bb{L}_{\mu}\Omega \lra \bb{S}_{\mu}\Omega \lra \mc{E}' \lra 0\]
    where $\mc{E}'$ has a filtration with composition factors $\bb{L}_{\ll}\Omega$, where $\ll<\mu$ in the dominance order. We have shown that Theorem~\ref{thm:coh-omega} holds for $\bb{S}_{\mu}\Omega$, and we know by induction that it holds for (each of the composition factors of) $\mc{E}'$, hence using the long exact sequence in cohomology it also holds for $\bb{L}_{\mu}\Omega$, concluding our proof.
\end{proof}

\subsection{Stabilization of cohomology} The goal of this section is to give a proof of Theorem~\ref{thm:stab-coh-omega}. We consider a hyperplane embedding $\iota:\PP^{n-1}\lra\PP^n$, and the associated short exact sequence
\begin{equation}\label{eq:ses-hyper-in-Pn} 0 \lra \mc{O}_{\PP^n}(-1) \lra \mc{O}_{\PP^n} \lra \iota_*\mc{O}_{\PP^{n-1}} \lra 0,
\end{equation}
as well as the relative cotangent sequence
\begin{equation}\label{eq:cot-seq-hyper-Pn} 0 \lra \mc{O}_{\PP^{n-1}}(-1) \lra \iota^*\Omega_{\PP^n} \lra \Omega_{\PP^{n-1}} \lra 0,
\end{equation}
which is split since $H^1(\PP^{n-1},\Omega_{\PP^{n-1}}^{\vee}(-1))=0$ (which follows by dualizing \eqref{eq:Euler-sequence}, twisting by $\mc{O}_{\PP^{n-1}}(-1)$, and considering the long exact sequence in cohomology).

\begin{proof}[Proof of Theorem~\ref{thm:stab-coh-omega}]
    We first verify (1) and (2). If we tensor \eqref{eq:ses-hyper-in-Pn} by $\mc{P}(\Omega_{\PP^n})$ and use the functorial property $\iota^*\mc{P}(\Omega_{\PP^n})=\mc{P}(\iota^*\Omega_{\PP^n})$ and the projection formula, then we get a short exact sequence
    \[0\lra \mc{P}\left(\Omega_{\PP^n}\right)\oo\mc{O}_{\PP^n}(-1) \lra \mc{P}\left(\Omega_{\PP^n}\right) \lra \iota_*\mc{P}\left(\iota^*\Omega_{\PP^n}\right) \lra 0 \]
    By Theorem~\ref{thm:coh-omega}(3), (4) with $e=-1$ (and $n-1$ replaced by $n$), we have that 
    \[H^j\left(\PP^n,\mc{P}\left(\Omega_{\PP^n}\right)\oo\mc{O}_{\PP^n}(-1)\right) = H^{j+1}\left(\PP^n,\mc{P}\left(\Omega_{\PP^n}\right)\oo\mc{O}_{\PP^n}(-1)\right) = 0\quad\text{ if }j<n-1,\text{ or if }d\leq n-1.\]
    The long exact sequence in cohomology induces then isomorphisms
    \begin{equation}\label{eq:cohPOm=cohistarPOm}
    H^j\left(\PP^n,\mc{P}(\Omega_{\PP^n})\right) \simeq H^j\left(\PP^{n-1},\mc{P}\left(\iota^*\Omega_{\PP^{n}}\right)\right)\quad\text{ if }j<n-1,\text{ or if }d\leq n-1.
    \end{equation}
    Using the split sequence \eqref{eq:cot-seq-hyper-Pn} and \cite{draisma}*{Lemma~14}, it follows that there exist polynomial functors $\mc{P}^{(i)}$ of degree $d-i$ such that
    \[ \mc{P}\left(\iota^*\Omega_{\PP^{n}}\right) = \mc{P}\left(\Omega_{\PP^{n-1}}\right) \oplus \left(\bigoplus_{i=1}^d \mc{P}^{(i)}\left(\Omega_{\PP^{n-1}}\right)\oo\mc{O}_{\PP^{n-1}}(-i)\right).\]
    For each $i=1,\cdots,d$, it follows from Theorem~\ref{thm:coh-omega}(3), (4) that 
    \[H^j\left(\PP^{n-1},\mc{P}^{(i)}\left(\Omega_{\PP^{n-1}}\right)\oo\mc{O}_{\PP^{n-1}}(-i)\right)=0\quad\text{ if }j<n-1,\text{ or if }d\leq n-1.\]
    Combining this with \eqref{eq:cohPOm=cohistarPOm}, it follows that
    \[H^j\left(\PP^n,\mc{P}(\Omega_{\PP^n})\right) \simeq H^j\left(\PP^{n-1},\mc{P}(\Omega_{\PP^{n-1}})\right)\quad\text{ if }j<n-1,\text{ or if }d\leq n-1,\]
    proving conclusions (1) and (2) of our theorem.
    
    To prove (3), we note that part (2) implies  $H^j\left(\PP^{n-1},\mc{P}(\Omega)\right)=H^j\left(\PP^{d},\mc{P}(\Omega)\right)$ for $n-1\geq d$. By Grothendieck vanishing, $H^j\left(\PP^{d},\mc{P}(\Omega)\right)=0$ for $j>d=\dim(\PP^d)$, hence (3) holds when $n-1\geq d$. If $n-1<d$ then we have $j>d>n-1=\dim(\PP^{n-1})$ and we conclude as before that $H^j\left(\PP^{n-1},\mc{P}(\Omega)\right)=0$.
\end{proof}

\subsection{K\"unneth Formula}\label{subsec:Kunneth}

The goal of this section is to prove Theorem~\ref{thm:Kunneth-coh-omega}, for which the main ingredient is Beilinson's resolution of the diagonal \cite{beilinson}. We write $\PP=\PP^{n-1}$, $V=H^0(\PP,\mc{O}_{\PP}(1))\simeq\kk^n$ and consider the tautological exact sequence (the twist of \eqref{eq:Euler-sequence} by $\mc{O}_{\PP}(1)$)
\[ 0 \lra \mc{R} \lra V \oo \mc{O}_{\PP} \lra \mc{O}_{\PP}(1) \lra 0\]
We write $\pi_1,\pi_2$ for the two projections $\PP\times\PP\lra\PP$, and write $\mc{E}\boxtimes\mc{F} = \pi_1^*\mc{E} \oo \pi_2^*\mc{F}$ for coherent sheaves $\mc{E},\mc{F}$ on $\PP$, noting that K\"unneth's formula holds for $\mc{E}\boxtimes\mc{F}$ \cite{stacks-project}*{\href{https://stacks.math.columbia.edu/tag/0BED}{Lemma 0BED}}. If we let $\Delta\subset\PP\times\PP$ denote the diagonal, then Beilinson constructs a resolution
\[ \mc{B}_{\bullet} \lra \mc{O}_{\Delta} \lra 0\]
where $\mc{B}_i = \bw^i \mc{R} \boxtimes \mc{O}_{\PP}(-i)$ for $i=0,\cdots,n-1$.

\begin{proof}[Proof of Theorem~\ref{thm:Kunneth-coh-omega}]
    If we use the identification of $\Delta$ with $\PP$, then we can realize the sheaf $\mc{P}(\Omega)\oo\mc{P}'(\Omega)$ as the restriction $\left(\mc{P}(\Omega)\boxtimes\mc{P}'(\Omega)\right)_{|_{\Delta}}$. We can then compute its cohomology as the hypercohomology of its resolution $\mc{B}_{\bullet} \oo \left(\mc{P}(\Omega)\boxtimes\mc{P}'(\Omega)\right)$. Since $\mc{B}_0 \oo \left(\mc{P}(\Omega)\boxtimes\mc{P}'(\Omega)\right) = \left(\mc{P}(\Omega)\boxtimes\mc{P}'(\Omega)\right)$ satisfies K\"unneth's formula, it suffices to prove that
    \[ H^j\left(\PP\times\PP,\mc{B}_i \oo \left(\mc{P}(\Omega)\boxtimes\mc{P}'(\Omega)\right)\right) = 0\text{ for }i>0\text{ and all }j. \]
    Applying again K\"unneth's formula, it suffices to verify that for each $i>0$ either the cohomology of $\mc{P}(\Omega)\oo\bw^i\mc{R}$ or that of $\mc{P}'(\Omega)\oo\mc{O}_{\PP}(-i)$ is identically zero. If $d'-n+1\leq -i$ then we can apply Theorem~\ref{thm:coh-omega}(3),(4) to get the desired vanishing for $\mc{P}'(\Omega)\oo\mc{O}_{\PP}(-i)$.

    To conclude the proof, we may then assume that $i\geq n-d'\geq d+1$ and prove that $\mc{P}(\Omega)\oo\bw^i\mc{R}$ has vanishing cohomology. To that end, we use the Koszul resolution
    \[ \mc{K}_{\bullet}\lra \bw^i\mc{R} \lra 0,\quad\text{ where }\quad\mc{K}_j = \bw^{i+1+j}V \oo \mc{O}_{\PP}(-j-1) \quad\text{ for }j=0,\cdots,n-i-1.\]
    Since $d-n+1\leq -n+i\leq -j-1\leq -1$ for $j$ as above, we conclude again by Theorem~\ref{thm:coh-omega}(3),(4) that $\mc{P}(\Omega)\oo\mc{K}_j$ has vanishing cohomology, so the same is true about $\mc{P}(\Omega)\oo\bw^i\mc{R}$, concluding our proof.
\end{proof}

\subsection{Truncated powers}\label{subsec:trunc-pows-Doty}

The goal of this section is to describe the (stable) cohomology for the truncated symmetric powers of $\Omega$.

\begin{proof}[Proof of \eqref{eq:stable-coh-trunc-pows}] For every short exact sequence
\[ 0 \lra \mc{L} \lra \mc{E} \lra \mc{F} \lra 0\]
where $\mc{E},\mc{F}$ are vector bundles and $\mc{L}$ is a line bundle, we have a resolution $\mc{G}_{\bullet}$ of $T_p\Sym^d\mc{F}$ given by
\[ \mc{G}_{2i} = \mc{L}^{pi}\oo T_p\Sym^{d-pi}\mc{E},\quad \mc{G}_{2i+1}=\mc{L}^{pi+1}\oo T_p\Sym^{d-pi-1}\mc{E},\quad\text{ for }i\geq 0.\]
Indeed, if $\rank(\mc{E})=e+1$ then on an affine chart $\op{Spec}(A)$ we can view $\mc{E}$ as a free $A$-module spanned by $x_0,\cdots,x_e$, $\mc{L} = A\cdot x_0 = A(-1)$ as a rank one free module generated in degree $1$, and $\mc{F}$ as a free $A$-module spanned by $x_1,\cdots,x_e$. We can then identify the truncated symmetric algebras of $\mc{E}$, $\mc{F}$ as follows: if we let
\[ B = A[x_0,\cdots,x_e]/\langle x_0^p,\cdots,x_e^p\rangle\]
then we get
\[ T_p\Sym(\mc{E}) = B\quad\text{ and }\quad T_p\Sym(\mc{F}) = B/\langle x_0\rangle=A[x_1,\cdots,x_e]/\langle x_1^p,\cdots,x_e^p\rangle.\]
Writing $B\cdot x_0^j = B(-j)$ for the rank $1$ free $B$-module generated in degree $j$, the minimal graded free resolution of $B/\langle x_0\rangle$ over $B$ is given by
\[\cdots \lra B(-pi-1) \overset{\cdot x_0}{\lra} B(-pi) \overset{\cdot x_0^{p-1}}{\lra} \cdots \overset{\cdot x_0}{\lra} B(-p) \overset{\cdot x_0^{p-1}}{\lra} B(-1) \overset{\cdot x_0}{\lra} B \lra B/\langle x_0\rangle \lra 0\]
This resolution globalizes and yields by restricting to degree $d$ the complex $\mc{G}_{\bullet}$.

Since the truncated powers $T_p\Sym^d$ are simple functors, they are invariant under the duality functor ${}^{\#}$, so
\[ T_p\Sym^d\left(\mc{E}^{\vee}\right) = \left(T_p\Sym^d\mc{E}\right)^{\vee}.\]
We write $V=\kk^n$, $n\gg 0$, and apply the above construction to the dual of the Euler sequence on $\PP=\PP^{n-1}$:
\[ 0 \lra \mc{O}_{\PP} \lra V^{\vee}\oo\mc{O}_{\PP}(1) \lra \Omega^{\vee} \lra 0.\]
Dualizing the resulting complex $\mc{G}_{\bullet}$ we get a right resolution $\mc{G}^{\vee}_{\bullet}$ of $T_p\Sym^d\Omega$ given by
\[ \mc{G}_{2i}^{\vee} = (T_p\Sym^{d-pi}V)\oo\mc{O}_{\PP}(pi-d),\quad \mc{G}^{\vee}_{2i+1}=(T_p\Sym^{d-pi-1}V)\oo\mc{O}_{\PP}(pi+1-d),\quad\text{ for }i\geq 0.\]
Since the sheaves $\mc{G}_{i}^{\vee}$ are isomorphic to direct sums of line bundles $\mc{O}_{\PP}(-j)$, $-d\leq j\leq 0$, and $n\gg 0$, it follows that the only ones that have non-zero cohomology occur for $j=0$. In particular, if $d\not\equiv 0,1\ (\op{mod }p)$ then $pi-d$ and $pi+1-d$ are never $0$, hence $T_p\Sym^d\Omega$ has no cohomology. If $d\equiv 0,1\ (\op{mod }p)$ then $\mc{G}_i^{\vee}$ has non-vanishing cohomology precisely when $i=|d|_p$, since in that case $\mc{G}_{|d|_p}^{\vee}=\mc{O}_{\PP}$. The hypercohomology spectral sequence associated to the resolution $\mc{G}^{\vee}_{\bullet}$ of $T_p\Sym^d\Omega$ now implies \eqref{eq:stable-coh-trunc-pows}, as desired.
\end{proof} 

\subsection{Cohomology of Weyl functors}

We next consider the case when $\mc{P}=\bb{W}_{\mu}$ is a Weyl functor associated to a partition $\mu$ of $d$. Identifying $\Omega = \ker(\phi)$ in \eqref{eq:Euler-sequence} and using the fact that $\phi$ is locally a split surjection, the dual version of \cite{ABW}*{Corollary~V.1.15} yields a right resolution of $\bb{W}_{\mu}\Omega$ given by the complex $\mc{G}^{\bullet}=\bb{W}_{\mu}\left(\mc{F}^0\overset{\phi}{\lra}\mc{F}^1\right)$, where $\mc{F}^0 = V \oo \mc{O}_\PP(-1)$ and $\mc{F}^1 = \mc{O}_\PP$. More precisely, we have
\begin{equation}\label{eq:def-Gi}
 \mc{G}^i = \bb{W}_{\mu/(1^i)}V \oo \mc{O}_{\PP}(-|\mu|+i)\quad \text{ for }i=0,\cdots,\mu'_1,
\end{equation}
and in particular each $\mc{G}^i$ is isomorphic to a direct sum of line bundles.

\begin{theorem}\label{thm:Weyl-Omega}
    If $\mu$ is a partition with $\mu_1\geq 2$ and $\mc{P}$ is a polynomial functor of degree $d$, then
    \begin{equation}\label{eq:Hj-WOmega-vanish-small-j} H^j\left(\PP,\mc{P}(\Omega)\oo\bb{W}_{\mu}\Omega\right) = 0\text{ for all }j<n-1.
    \end{equation}
    Moreover, if $d+|\mu|\leq n-1$ then we also have $H^{n-1}\left(\PP,\mc{P}(\Omega)\oo\bb{W}_{\mu}\Omega\right) = 0$.
\end{theorem}

\begin{proof} Tensoring the resolution \eqref{eq:def-Gi} with $\mc{P}(\Omega)$, we get a hypercohomology spectral sequence
\[ E_2^{i,j} = H^j\left(\PP,\mc{P}(\Omega)\oo\mc{G}^i\right) \Longrightarrow H^{i+j}\left(\PP,\mc{P}(\Omega)\oo\bb{W}_{\mu}\Omega\right).\]
Since $\mu_1\geq 2$, we have $-|\mu|+i<0$ for all $i=0,\cdots,\mu'_1$, and therefore the sheaves $\mc{P}(\Omega)\oo\mc{G}^i$ may only have cohomology in degree $j=n-1$ by Theorem~\ref{thm:coh-omega}(3), proving \eqref{eq:Hj-WOmega-vanish-small-j}. If in addition $d+|\mu|\leq n-1$, then
\[ d-n+1 \leq -|\mu|+i \leq -1 \quad\text{ for }i=0,\cdots,\mu'_1,\]
and Theorem~\ref{thm:coh-omega}(4) implies $E_2^{i,j}=0$ for all $i,j$. We conclude that in this case the cohomology of $\mc{P}(\Omega)\oo\bb{W}_{\mu}\Omega$ is identically zero, as desired.
\end{proof}

If we consider instead partitions with $\mu_1=1$ then $\mu=(1^m)$ for some $m$, and $\bb{W}_{\mu}\Omega=\Omega^m$. Using \eqref{eq:Hd-of-Omegad}, it is then clear that Theorem~\ref{thm:Weyl-Omega} fails already when $\mc{P}$ is the trivial functor. In fact, \eqref{eq:Hd-of-Omegad} together with K\"unneth's formula implies that tensoring with $\Omega^m$ induces multiplication by $t^m$ on the stable cohomology polynomials. We prove the following more precise statement.

\begin{theorem}\label{thm:tens-Omega=shift-coh}
 If $\mc{P}$ is a polynomial functor of degree $d$ then there is a natural identification
 \[ H^{j+m}\left(\PP,\mc{P}(\Omega)\oo\Omega^m\right) = H^j\left(\PP,\mc{P}(\Omega)\right)\quad\text{for }j+m<n-1.\]
 Moreover, the above identification also holds when $j+m=n-1$ provided $d\leq j$. 
\end{theorem}

\begin{proof}
   We argue by induction on $m$, noting that if $m=0$ then there is nothing to prove. If $m>0$ then \eqref{eq:Euler-sequence} induces a short exact sequence
   \[ 0 \lra \Omega^m \lra \bw^m V \oo \mc{O}_{\PP}(-m) \lra \Omega^{m-1} \lra 0.\]
   We tensor with $\mc{P}(\Omega)$ and consider the long exact sequence in cohomology
   \[\begin{aligned}\cdots\lra &\bw^m V \oo H^{i-1}\left(\PP,\mc{P}(\Omega)\oo \mc{O}_{\PP}(-m)\right) \lra H^{i-1}\left(\PP,\mc{P}(\Omega)\oo\Omega^{m-1}\right) \overset{\delta}{\lra} H^i\left(\PP,\mc{P}(\Omega)\oo\Omega^m\right) \lra \\
   \lra &\bw^m V \oo H^i\left(\PP,\mc{P}(\Omega)\oo \mc{O}_{\PP}(-m)\right) \lra \cdots
   \end{aligned}\]
   If $i=j+m<n-1$ then it follows from Theorem~\ref{thm:coh-omega}(3) that 
   \[H^{i-1}\left(\PP,\mc{P}(\Omega)\oo \mc{O}_{\PP}(-m)\right)=H^{i}\left(\PP,\mc{P}(\Omega)\oo \mc{O}_{\PP}(-m)\right)=0,\] 
 hence the connecting map $\delta$ is an isomorphism. The desired conclusion now follows by induction on $m$. If $i=j+m=n-1$ then we still get that $\delta$ is injective using Theorem~\ref{thm:coh-omega}(3). If $d\leq j$ then it follows that $d-n+1\leq j-n+1=-m\leq -1$, hence Theorem~\ref{thm:coh-omega}(4) implies $H^i\left(\PP,\mc{P}(\Omega)\oo \mc{O}_{\PP}(-m)\right)=0$ and proves that $\delta$ is also surjective.
\end{proof}

\begin{remark}\label{rem:case-d<j}
The last assertion in Theorem~\ref{thm:tens-Omega=shift-coh} is interesting only when $d=j$: if $d<j$ then we know already by Theorem~\ref{thm:stab-coh-omega}(3) that both $H^{j+m}\left(\PP,\mc{P}(\Omega)\oo\Omega^m\right)$ and $H^j\left(\PP,\mc{P}(\Omega)\right)$ vanish.
\end{remark}

Unlike the examples from Section~\ref{subsec:trunc-pows-Doty} where the cohomology of $T_p\Sym^d\Omega$ was non-zero in at most one degree, this is no longer the case for $\bb{L}_{\ll}\Omega$ when $\ll$ is a general partition, as explained next.

\begin{example}\label{ex:simpleWithMoreCohom}
    In this example, we show that if $\chr(\kk)=2$ then $\bb{L}_{(2,2,1,1)} \Omega$ satisfies
$$H^4_{st}(\bb{L}_{(2,2,1,1)} \Omega ) = H^5_{st} (\bb{L}_{(2,2,1,1)} \Omega) = \kk,$$
and all other cohomology vanishes. To see this, we can use the GAP package \cite{doty-gap} to find that there is a short exact sequence
$$0 \to \bb{L}_{(1^6)} (\Omega) \to \bbw_{(2,1^4)} (\Omega) \to \bb{L}_{(2,1^4)} (\Omega) \to 0.$$
Since $\bb{L}_{(1^6)} (\Omega) = \Omega^6$ and $\bbw_{(2,1^4)} (\Omega)$ has vanishing stable cohomology by Theorem~\ref{thm:Weyl-Omega}, we get $h_{\bb{L}_{(2,1^4)}} (t) = t^5$. We have moreover short exact sequences
$$0 \to \mc{P}(\Omega) \to \bbw_{(2,2,1,1)} (\Omega) \to \bb{L}_{(2,2,1,1)} (\Omega) \to 0,$$
$$0 \to \bb{L}_{(2,1^4)} (\Omega)  \to \mc{P}(\Omega) \to \bb{L}_{(1^6)} (\Omega) \to 0,$$
for some polynomial functor $\mc{P}$. From the latter it follows that $h_{\mc{P}(\Omega)} (t) = t^5 + t^6$ and hence by the former (combined with the fact that $\bbw_{(2,2,1,1)} (\Omega)$ has vanishing stable cohomology) we conclude that
$$h_{\bb{L}_{(2,2,1,1)}} (t) = t^4 + t^5.$$
\end{example}

The following will be crucial in the calculation of stable cohomology for Schur functors associated to hook shapes in Section~\ref{sec:arithmetic}.

\begin{corollary}\label{cor:multOmega-iso-cohom}
    Suppose that $\mc{P}$ is a polynomial functor of degree $d$. 
    \begin{itemize}
    \item[(i)] The exterior multiplication $\Omega^a\oo\Omega^b\lra\Omega^{a+b}$ induces an isomorphism in cohomology
  \[H^j\left(\PP,\mc{P}(\Omega)\oo\Omega^a \oo \Omega^b \right) = H^j\left(\PP,\mc{P}(\Omega)\oo\Omega^{a+b} \right)\text{ for }j<n-1.\]
  If $d\leq n-1-a-b$ then the above isomorphism also holds for $j=n-1$.
  \item[(ii)] Under the identification in (i), the comultiplication $\Delta:\Omega^{a+b}\lra\Omega^a\oo\Omega^b$ induces a map in cohomology
  \[ H^j(\PP,\Delta):H^j\left(\PP,\mc{P}(\Omega)\oo\Omega^{a+b} \right)\lra H^j\left(\PP,\mc{P}(\Omega)\oo\Omega^a \oo \Omega^b \right) = H^j\left(\PP,\mc{P}(\Omega)\oo\Omega^{a+b} \right)\]
  which is multiplication by ${a+b\choose b}$.
\end{itemize}
\end{corollary}

\begin{proof}
    To prove (i), suppose that $a\geq b$ and consider the short exact sequence
\[0 \lra \bb{W}_{\ll/(1^{b-1})}\Omega \lra \Omega^a \oo \Omega^b \lra \Omega^{a+b} \lra 0,\]
where $\ll = (a+b-1,b)'$. Tensoring with $\mc{P}(\Omega)$ and using the fact that $\bb{W}_{\ll/(1^{b-1})}\Omega$ has a filtration with composition factors of the form $\bb{W}_{\mu}\Omega$ with $\mu_1=2$, it follows from Theorem~\ref{thm:Weyl-Omega} and the long exact sequence in cohomology that the induced map
\begin{equation}\label{eq:ind-map-Hj-PO-oo-Omegas}
H^j\left(\PP,\mc{P}(\Omega)\oo\Omega^a \oo \Omega^b \right) \lra H^j\left(\PP,\mc{P}(\Omega)\oo\Omega^{a+b} \right)
\end{equation}
is injective for $j<n-1$. It follows from Theorem~\ref{thm:tens-Omega=shift-coh} that the above cohomology groups are abstractly isomorphic to $H^{j-a-b}\left(\PP,\mc{P}(\Omega)\right)$ (which is a finite dimensional vector space), hence \eqref{eq:ind-map-Hj-PO-oo-Omegas} is an isomorphism. Since $|\mu|=a+b$, the same argument applies for $j=n-1$ using the last assertions in Theorems~\ref{thm:Weyl-Omega},~\ref{thm:tens-Omega=shift-coh}.

Part (ii) follows from the fact that the composition of the comultiplication map $\Omega^{a+b}\lra\Omega^a\oo\Omega^b$ with the exterior multiplication $\Omega^a\oo\Omega^b\lra\Omega^{a+b}$ is scalar multiplication by ${a+b\choose b}$.
\end{proof}

\subsection{Invariance under Frobenius}\label{subsec:invariance-Frob} 

We begin by considering precomposition with Frobenius.

\begin{lemma}\label{lem:Frob-PcircFp}
    Suppose that $\mc{P}$ is a polynomial functor of degree $d$. If $n-1\geq dp$ then
    \[H^j\left(\PP^{n-1},(\mc{P}\circ F^p)(\Omega)\right) = H^j\left(\PP^{n-1},\mc{P}(\Omega)\right).\]
\end{lemma}

\begin{proof}
If we write $V=\kk^n$, $\bb{P}(V)=\op{Proj}(\Sym V)$, then the Frobenius morphism is naturally defined as
\[\varphi:\PP^{n-1}\simeq\bb{P}(V) \lra \bb{P}(F^pV)\simeq\PP^{n-1},\]
so that we can identify $F^p\Omega = \varphi^*(\Omega)$, and therefore
\[ \mc{P}(F^p(\Omega)) = \mc{P}(\varphi^*(\Omega)) = \varphi^*(\mc{P}(\Omega)).\]
Since $\varphi$ is a finite morphism, it follows from the projection formula that
    \[ H^j\left(\PP^{n-1},\varphi^*(\mc{P}(\Omega))\right) = H^j\left(\PP^{n-1},\mc{P}(\Omega) \oo \varphi_*(\mc{O}_{\PP^{n-1}})\right)\quad\text{for all }j.\]
    Using the decomposition (see for instance \cite{achinger}*{Theorem~2})
    \[\varphi_*(\mc{O}_{\PP^{n-1}}) = \mc{O}_{\PP^{n-1}} \oplus \bigoplus_{i=1}^{\lfloor n(p-1)/p\rfloor} \mc{O}_{\PP^{n-1}}(-i)^{\oplus a_i},\]
    it suffices to prove that $\mc{P}(\Omega) \oo \mc{O}_{\PP^{n-1}}(-i)$ has vanishing cohomology for $1\leq i\leq \lfloor n(p-1)/p\rfloor$. Since $n\geq dp+1$, we have $n(p-1)/p<n-d$, hence $i\leq n-d-1$. The desired vanishing now follows by taking $e=-i$ in Theorem~\ref{thm:coh-omega}(3),(4), concluding our proof.
\end{proof}

If $\mc{P}$ is a polynomial functor defined over the prime field $\bb{Z}/p\bb{Z}$, we can identify $\mc{P}\circ F^p=F^p\circ\mc{P}$, so Theorem~\ref{thm:invariance-Frob} follows from Lemma~\ref{lem:Frob-PcircFp} for such $\mc{P}$. To address the general case, we first record some elementary consequences of the K\"unneth formula and our results on the cohomology for Weyl modules (see also \cite{RV}*{Corollary~3.2}). For a sequence $\ul{d}=(d_1,\cdots,d_r)$ with $d_1+\cdots+d_r=d$, it follows from Theorem~\ref{thm:Weyl-Omega} that
\begin{equation}\label{eq:Hst-Ddd}
 H^j_{st}\left(D^{\ul{d}}\Omega\right) = \begin{cases}
  \kk & \text{if }j=d\text{ and }\ul{d}=(1^d), \\
  0 & \text{otherwise.}
 \end{cases}
\end{equation}
Since $D^{\ul{d}}$ is defined over $\bb{Z}/p\bb{Z}$, we can apply Lemma~\ref{lem:Frob-PcircFp} to get
\begin{equation}\label{eq:Hst-Frob-Ddd}
 H^j_{st}\left(F^p(D^{\ul{d}}\Omega)\right) = H^j_{st}\left(D^{\ul{d}}(F^p\Omega)\right) = \begin{cases}
  \kk & \text{if }j=d\text{ and }\ul{d}=(1^d), \\
  0 & \text{otherwise.}
 \end{cases}
\end{equation}
When $\ul{d}=(1^d)$ we have that $D^{(1^d)}=T^d$ is the tensor power functor, which has a natural action of the symmetric group $\mf{S}_d$ by permuting the tensor factors. We have an identification
\[ \Hom_{\mf{Pol}_d}(T^d,T^d) = \kk[\mf{S}_d],\]
and for a permutation $\sigma\in\mf{S}_d$ we get that the induced map on stable cohomology
\[ H^d_{st}(\sigma) = \left(\kk \overset{\cdot\sgn(\sigma)}{\lra}\kk \right)\]
is scalar multiplication by the sign of $\sigma$. We abbreviate this as $H^d_{st}(\sigma)=\sgn(\sigma)$, and using the identification
\[ \Hom_{\mf{Pol}_d}(T^d,T^d) = \kk[\mf{S}_d],\]
we can write every endomorphism $\phi$ of $T^d$ as $\phi=\sum_{\s\in\mf{S}_d}a_{\s}\cdot\sigma$, thus $H^d_{st}(\phi) =  \sum_{\s\in\mf{S}_d}a_{\s}\cdot\sgn(\sigma)$. Postcomposition with $F^p$ yields an endomorphism
\[\mathstrut^{p}\phi\in\Hom_{\mf{Pol}_d}(F^p\circ T^d,F^p\circ T^d),\quad\text{given by}\quad \mathstrut^{p}\phi=\sum_{\s\in\mf{S}_d}a_{\s}^p\cdot\sigma,\]
Applying the stable cohomology functor $H^d_{st}(-)$ yields scalar multiplication by
\begin{equation}\label{eq:Hdst-phip=pphi}
H^d_{st}\left(\mathstrut^{p}\phi\right) = \sum_{\s\in\mf{S}_d}a_{\s}^p\cdot\sgn(\sigma) = \sum_{\s\in\mf{S}_d}a_{\s}^p\cdot\sgn(\sigma)^p = \left(\sum_{\s\in\mf{S}_d}a_{\s}\cdot\sgn(\sigma)\right)^p.
\end{equation}

\begin{proof}[Proof of Theorem~\ref{thm:invariance-Frob}] By Lemma~\ref{lem:Frob-PcircFp}, it suffices to consider the composition $F^p\circ\mc{P}$, which is a polynomial functor of degree $pd$. It follows from Theorem~\ref{thm:stab-coh-omega} and the hypothesis $n-1\geq p\cdot d$ that in order to complete the proof of Theorem~\ref{thm:invariance-Frob}, we need to identify the stable cohomology of $\mc{P}(\Omega)$ and $F^p(\mc{P}(\Omega))$. The functor $\mc{P}$ admits a resolution $\mc{F}_{\bullet}:\cdots\lra\mc{F}_1\lra\mc{F}_0$, where each $\mc{F}_i$ is a direct sum of $D^{\ul{d}}$ with $|\ul{d}|=d$ \cite{Krause-book}*{Chapter~8}. 

It follows from \eqref{eq:Hst-Ddd} that the stable cohomology of $\mc{P}(\Omega)$ is computed as the homology of the complex $H^d_{st}\left(\mc{F}_{\bullet}(\Omega)\right)$, where the only non-trivial contributions to stable cohomology come from the summands $D^{\ul{d}}$ of $\mc{F}_i$ for which $\ul{d}=(1^d)$. Similarly, \eqref{eq:Hst-Frob-Ddd} implies that the stable cohomology of $F^p(\mc{P}(\Omega))$ is computed as the homology of the complex $H^d_{st}\left((F^p\circ\mc{F}_{\bullet})(\Omega)\right)$, and moreover $H^d_{st}\left((F^p\circ\mc{F}_{i})(\Omega)\right)=H^d_{st}\left(\mc{F}_{i}(\Omega)\right)$ for all $i$. Using \eqref{eq:Hdst-phip=pphi}, we obtain that the differentials in the complex $H^d_{st}\left((F^p\circ\mc{F}_{\bullet})(\Omega)\right)$ differ from those of $H^d_{st}\left(\mc{F}_{\bullet}(\Omega)\right)$ by a Frobenius twist. Since Frobenius is an exact functor on the category of $\kk$-vector spaces, it follows that $H^j_{st}\left(F^p(\mc{P}(\Omega))\right) = H^j_{st}\left(\mc{P}(\Omega)\right)$ for all $j$, concluding our proof.
\end{proof}

\subsection{Vanishing for $p$-cores with multiple columns}

Fix a $p$-core $\gamma$ and consider the subcategory $\mf{Pol}_{d,\gamma}$ of $\mf{Pol}_d$ consisting of functors which admit a filtration with composition factors $\bb{L}_{\mu}$ where $\op{core}_p(\mu)=\gamma$. Note that for this to be non-empty, we must have $|\gamma|\leq d$ and $|\gamma|\equiv d\ (\op{mod }p)$.

\begin{theorem}\label{thm:vanishing-p-core}
    If $\gamma_1\geq 2$ and $\mc{P}\in\mf{Pol}_{d,\gamma}$ then $H^j_{st}(\mc{P}(\Omega))=0$ for all $j$.
\end{theorem}

\begin{proof}
    It suffices to prove the theorem for simple functors $\mc{P}=\bb{L}_{\ll}$ where $\op{core}_p(\ll)=\gamma$, and in particular $\ll_1\geq 2$. We argue by induction on $\ll$, with respect to the dominance order. If $\ll$ is minimal then $\bb{L}_{\ll}=\bb{W}_{\ll}$, and the conclusion follows from Theorem~\ref{thm:Weyl-Omega}. Otherwise, we have a short exact sequence in $\mf{Pol}_{d,\gamma}$
    \[ 0\lra \mc{P}' \lra \bb{W}_{\ll} \lra \bb{L}_{\ll} \lra 0,\]
    where $\mc{P}'$ has a filtration with composition factors $\bb{L}_{\mu}$ with $\mu<\ll$. By induction we know that the stable cohomology of $\mc{P}'\Omega$ vanishes, and by Theorem~\ref{thm:Weyl-Omega} we know that the stable cohomology of $\bb{W}_{\ll}\Omega$ vanishes as well. It follows that the same is true for $\bb{L}_{\ll}\Omega$, concluding our induction step and the proof.
\end{proof}

\section{Akin--Buchsbaum resolutions by tensor products of exterior powers}\label{sec:resolutions-ext-pows}

The goal of this section is to describe a concrete recipe to compute the stable cohomology of $\mc{P}\Omega$ for certain Schur functors $\mc{P}$, by employing explicit resolutions of $\mc{P}$ by tensor products of exterior power functors constructed based on \cite{AB1}. More precisely, we will consider two special cases:
\begin{itemize}
    \item $\mc{P}=\bb{S}_{\ll/\mu}$ is a skew-Schur functor associated to a ribbon shape (a special case of \cite{AB1}*{Section~6}).
    \item $\mc{P}=\bb{S}_{\ll}$ is the Schur functor associated to a partition with at most two columns \cite{AB1}*{Section~4}.
\end{itemize}
The explicit constructions of these resolutions give rise via hypercohomology to concrete complexes of finite dimensional vector spaces whose homology describes the stable cohomology of $\mc{P}\Omega$. In later sections we will give a more in-depth study of the homology of these complexes, leading to some very explicit calculations of stable cohomology. 

\subsection{Ribbon shapes}\label{subsec:ribbons}

Given partitions $\ll,\mu$, we write $\mu\subseteq\ll$ if $\mu_i\leq \ll_i$ for all $i$, and define the \defi{skew-shape} $\ll/\mu$ to be the diagram obtained from $\ll$ by removing the first $\mu_i$ boxes in row $i$ for each $i$. For instance, if $\ll = (4,3,3,2)$ and $\mu=(2,2,1)$ we get
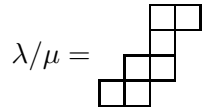
\[
\ytableausetup{centertableaux,smalltableaux}
\ll/\mu = \ydiagram{2+2,2+1,1+2,2}
\]
We say that $\ll/\mu$ is a \defi{ribbon} if it is connected ($\mu_i<\ll_{i+1}$ whenever $\ll_{i+1}>0$), such as the displayed example above. The \defi{size} $|\ll/\mu|$ is given by $\sum(\ll_i-\mu_i)$, and it will often be convenient to record a ribbon of size $m$ by a composition $w_0+\cdots+w_d$ of $m$, where the parts $w_i$ represent (from left to right, and starting the column count at $0$) the number of boxes in the columns of $\ll/\mu$. In the example above, the composition $\ul{w}$ corresponding to $\ll/\mu$ is $\ul{w}=(1,2,3,1)$. We refer the reader to \cites{BTW,las-pra} for a more thorough introduction to ribbons and the combinatorial study of associated Schur functions. We note that our conventions are such that if $e_i$ denotes the $i$-th elementary symmetric function (with $e_0=1$ and $e_i=0$ for $i<0$), and if $\ll/\mu$ is the ribbon with composition $\ul{w}$, then the ribbon Schur function associated to $\ll/\mu$ is given by the Jacobi-Trudi formula
\begin{equation}\label{eq:Jac-Trudi-ribbon} s_{\ll/\mu} = \det(e_{\ll'_i-\mu'_j-i+j}) = \det\begin{bmatrix}
   e_{w_0} & e_{w_0+w_1} & e_{w_0+w_1+w_2} & \cdots & e_{w_0+\cdots + w_d} \\
   1 & e_{w_1} & e_{w_1+w_2} & \cdots & e_{w_1+\cdots+w_d} \\
   0 & 1 & e_{w_2} & \cdots & e_{w_2+\cdots+w_d} \\
   \vdots & \vdots & \vdots & \ddots & \vdots \\
   0 & 0 & 0 & \cdots & e_{w_d}
\end{bmatrix}.
\end{equation}
In characteristic zero, it is well-known that such (alternating sign) character formulas are often explained as an equivariant Euler characteristic calculation associated to a natural resolution of a given representation (the BGG resolution is perhaps the most famous instance of this \cite{BGG}, see also \cite{zel-res} and \cites{akin1,akin2}). In \cite{AB1}, Akin and Buchsbaum initiated the study of resolutions of Schur functors by tensor products of exterior powers in arbitrary characteristic. Although the resolutions are not explicit for a general skew-shape $\ll/\mu$, they are so in the case of ribbons where they precisely reflect the combinatorial identity \eqref{eq:Jac-Trudi-ribbon}, as we explain next.

Given a tuple $\ul{w}=(w_0,\cdots,w_d)$, a \defi{refinement} of $\ul{w}$ is obtained by adding together consecutive components of $\ul{w}$. Such refinements are naturally parametrized by subsets $J\subseteq [d]=\{1,\cdots,d\}$, and can be visualized as we show next by thinking of $[d]$ as indexing edges, and of $w_i$ as weights on the vertices of the path graph
\[ 
\xymatrix{ \overset{w_0}{\bullet} \ar@{-}[r]^1 & \overset{w_1}{\bullet} \ar@{-}[r]^2 & \overset{w_2}{\bullet} \ar@{.}[r] & \overset{w_{d-1}}{\bullet} \ar@{-}[r]^d & \overset{w_d}{\bullet} }
\]
Every subset $J\subseteq[d]$ determines a subgraph $\Gamma_J$ of the path graph consisting of disjoint intervals (connected components), and we define the \defi{weight} of an interval to be the sum of weights of the vertices it contains. We write $\ul{w}^J$ for the refinement of $\ul{w}$ whose entries are the weights of the intervals determined by $J$. For instance, we have for $d=7$ and $J = \{1,2,3,6\}$ we have $\ul{w}^J = (w_0+w_1+w_2+w_3,w_4,w_5+w_6,w_7)$ with corresponding graph (we usually remove the vertex and edge labels when they are understood)
\[
\Gamma_J:\qquad \xymatrix{ \bullet \ar@{-}[r] & \bullet \ar@{-}[r] & \bullet \ar@{-}[r] & \bullet  & \bullet & \bullet \ar@{-}[r] & \bullet & \bullet}
\]

For a tuple $\ul{w}$ we consider the polynomial functor
\[ \bw^{\ul{w}} = \bw^{w_0} \oo \cdots \oo \bw^{w_d}\]
and we define a complex $\mf{F}_{\bullet}=\mf{F}_{\bullet}^{\ul{w}}$ in the category of polynomial functors whose terms are
\[\mf{F}_k = \bigoplus_{J\in{[d]\choose k}} \mf{F}_{J},\quad\text{ where }\quad\mf{F}_{J} = \bw^{\ul{w}^J}.\]
To define the differential $\pd_k$ we write it as $\pd_k = \bigoplus \pd_{J,J'}$, where $\pd_{J,J'}:\mf{F}_{J}\lra\mf{F}_{J'}$ for $J=\{j_1<\cdots<j_k\}$ and $J' = J\setminus\{j_i\}$. To define $\pd_{J,J'}$, we write $\ul{w}^J=(w^J_0,w^J_1,\cdots)$, $\ul{w}^{J'}=(w^{J'}_0,w^{J'}_1,\cdots)$, and note that the removal of the edge $j_i$ from $\Gamma_J$ disconnects precisely one interval in $\Gamma_J$ of weight $w^J_s$ into two intervals of weights $w^{J'}_s$ and $w^{J'}_{s+1}$. We let
\begin{equation} \label{eq:pd-JJ'-on-F}
 \pd_{J,J'} = (-1)^{s}\cdot \left(\Delta \oo \op{id}\right), 
\end{equation}
where $\Delta$ is the comultiplication map
\[ \Delta:\bw^{w^J_s} \lra \bw^{w^{J'}_s} \oo \bw^{w^{J'}_{s+1}},\]
and $\op{id}$ denotes the identity map on the remaining factors of $\mf{F}_J$ and $\mf{F}_{J'}$. Continuing with the earlier example $J=\{1,2,3,6\}$, if we take $i=3$ then $j_i=3$, $J' = \{1,2,6\}$, and therefore we have $\ul{w}^{J'} = (w_0+w_1+w_2,w_3,w_4,w_5+w_6,w_7)$ and
\[ \bw^{w_0+w_1+w_2+w_3} \oo \left(\bw^{w_4} \oo \bw^{w_5+w_6} \oo \bw^{w_7}\right) \overset{\pd_{J,J'}=\Delta\oo\op{id}}{\lra} \left(\bw^{w_0+w_1+w_2} \oo \bw^{w_3}\right) \oo \left(\bw^{w_4} \oo \bw^{w_5+w_6} \oo \bw^{w_7}\right) \]
The construction from \cite{AB1}*{Section~6} specialized to the case of ribbons gives the following.

\begin{theorem}\label{thm:res-Schur-ribbons}
 If $\ul{w}$ is the composition corresponding to the ribbon $\ll/\mu$ then the complex $\mf{F}_{\bullet}^{\ul{w}}$ is a resolution of the ribbon Schur functor $\bb{S}_{\ll/\mu}$.
\end{theorem}

\begin{proof}
When $d=0$ we have $\bb{S}_{\ll/\mu}=\bw^{w_0}$ and there is nothing to prove. Suppose now $d>0$ and observe that the comultiplication maps
 \[ \bw^{w_0+w_1+\cdots + w_r} \lra \bw^{w_0} \oo \bw^{w_1+\cdots + w_r}\]
 induce a morphism of complexes
 \[ f : \mf{F}_{\bullet}^{(w_0+w_1,w_2,\cdots,w_d)} \lra \bw^{w_0} \oo\, \mf{F}_{\bullet}^{(w_1,w_2,\cdots,w_d)}\]
 whose mapping cone is $\mf{F}_{\bullet}^{\ul{w}}$. Using \cite{AB1}*{Equation (6) in Section~6} and induction on $d$, we get that
 \[ H_0\left(\mf{F}_{\bullet}^{(w_0+w_1,w_2,\cdots,w_d)}\right)\overset{H_0(f)}{\lra} \bw^{w_0} \oo\, H_0\left(\mf{F}_{\bullet}^{(w_1,w_2,\cdots,w_d)}\right)\]
 gives a presentation of $\bb{S}_{\ll/\mu}$. Using the mapping cone long exact sequence in homology, it follows that $\mf{F}_{\bullet}^{\ul{w}}$ resolves $\bb{S}_{\ll/\mu}$, as desired.
\end{proof}

\begin{remark}
    The reader familiar with the theory of Koszul algebras will recognize the complex $\mf{F}_\bullet$ as a homogeneous strand of the (coaugmented) cobar complex on the coalgebra $\bigwedge^\bullet$. It is a consequence of Backelin's work \cite{backelinThesis} and the Koszul property of the exterior algebra $\bigwedge^\bullet$ that $\mf{F}_\bullet$ is a resolution (see \cite{polishchuk2005quadratic}*{Proposition 8.3} for a dual statement), and the presentation of Schur modules from \cite{ABW} shows that the $0$th homology agrees with the ribbon Schur functor $\bbs_{\ll / \mu}$.
\end{remark}

\begin{remark}
A combinatorial shadow of the construction in the proof of Theorem~\ref{thm:res-Schur-ribbons} arises by considering the cofactor expansion of \eqref{eq:Jac-Trudi-ribbon} along the first column:
\[s_{\ll/\mu} = e_{w_0}\cdot\det\begin{bmatrix}
e_{w_1} & e_{w_1+w_2} & \cdots & e_{w_1+\cdots+w_d} \\
1 & e_{w_2} & \cdots & e_{w_2+\cdots+w_d} \\
\vdots & \vdots & \ddots & \vdots \\
0 & 0 & \cdots & e_{w_d}
\end{bmatrix}
-
\det\begin{bmatrix}
e_{(w_0+w_1)} & e_{(w_0+w_1)+w_2} & \cdots & e_{(w_0+w_1)+\cdots + w_d} \\
1 & e_{w_2} & \cdots & e_{w_2+\cdots+w_d} \\
\vdots & \vdots & \ddots & \vdots \\
0 & 0 & \cdots & e_{w_d}
\end{bmatrix}
\]
The reader can check that the first term in the above difference corresponds to $\bw^{w_0}\oo\mf{F}_{\bullet}^{(w_1,\cdots,w_d)}$, while the second term corresponds to $\mf{F}_{\bullet}^{(w_0+w_1,\cdots,w_d)}$.
\end{remark}

We next explain how the resolution $\mf{F}_{\bullet}$ gives rise to a universal complex computing the stable cohomology of 
$\bb{S}_{\ll/\mu}\Omega$, where $\ll/\mu$ is the ribbon Schur functor associated to the composition $\ul{w}$.
We work on projective space $\PP\simeq\PP^{n-1}$, and assume that $n\gg 0$. If we write $N=|\ul{w}|=|\ll|-|\mu|$ then we get from Theorem~\ref{thm:tens-Omega=shift-coh} and \eqref{eq:Hd-of-Omegad} that for every subset $J\subset[d]$ we have
\begin{equation}\label{eq:Hi-FJOmega}
H^i(\PP,\mf{F}_J(\Omega)) = \begin{cases} \kk & \text{if }i=N, \\
0 & \text{otherwise}.
\end{cases}
\end{equation}
We consider the complex of $\kk$-vector spaces $C_{\bullet}=C_{\bullet}(w_0,\cdots,w_d)$ defined by
\[ C_{\bullet} = H^N\left(\PP,\mf{F}_\bullet(\Omega)\right),\]
and write $C_J = H^N(\PP,\mf{F}_J(\Omega))=\kk\cdot f_J$, where $f_J$ denotes a fixed generator of $C_J$. It follows that 
\[C_k = \bigoplus_{J\in{[d]\choose k}} C_J \simeq \kk^{\oplus{d\choose k}}\text{ for all }k.\]
Since there is no danger of confusion, we denote by $\pd$ the differential in $C_{\bullet}$ (as we do for $\mf{F}_{\bullet}$), and we write $\pd_{J,J'}:C_J \lra C_{J'}$ for the restriction of $\pd$ to the summands $C_J\subset C_k$ and $C_{J'}\subset C_{k-1}$. It follows from \eqref{eq:pd-JJ'-on-F} and Corollary~\ref{cor:multOmega-iso-cohom}(ii) that $\pd_{J,J'}:C_J\lra C_{J'}$ is given by
\begin{equation}\label{eq:bin-coe-delJJ'} \pd_{J,J'}(f_J) = (-1)^s {w^J_s \choose w^{J'}_s}\cdot f_{J'}
\end{equation}
In the case $d=3$ the resulting complex is depicted in Figure~\ref{fig:Cw-complex}, where the graph $\Gamma_J$ takes the place of the component $C_J$ (note that $s$ counts the number of missing edges in $\Gamma_J$ which are to the left of the edge $j_i$, where $\{j_i\}=J\setminus J'$).

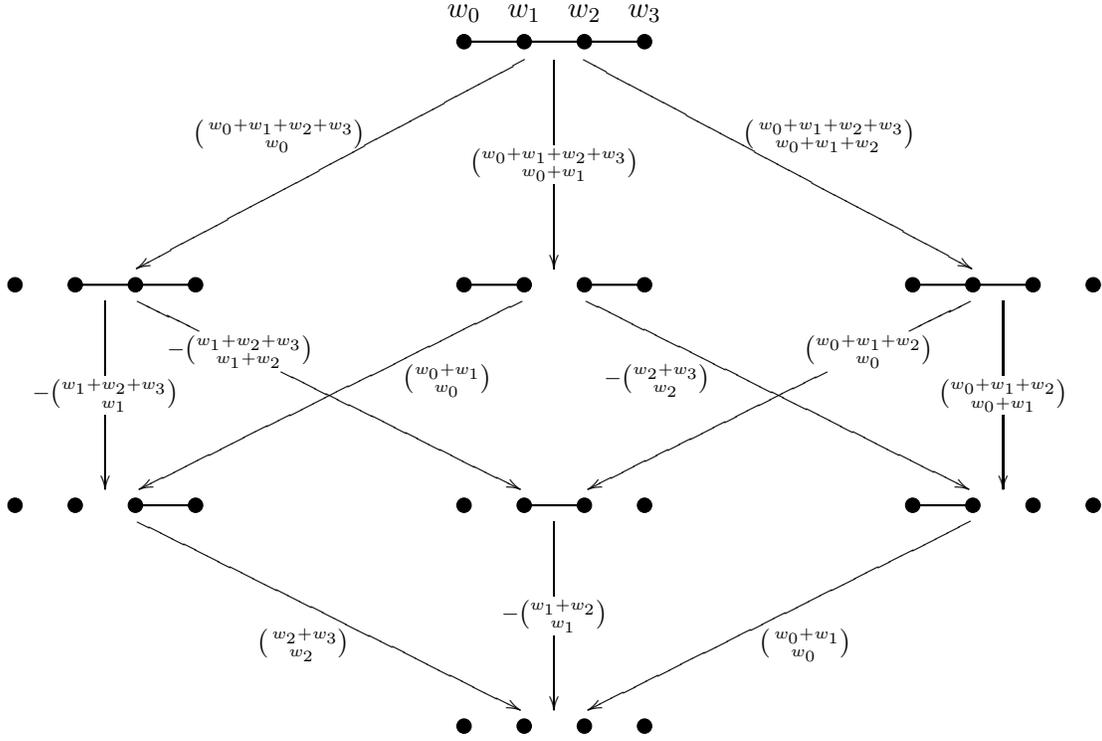
\begin{figure}[h]
\[
\xymatrix @C=7.5em @R=6.5em{
 & {\begin{tikzpicture}[scale=0.8]
\draw [thick] (0,0)--(1,0)--(2,0)--(3,0);
\node[arn_n] at (0,0) {};
\node[label=above:{$w_0$}] at (0,0) {};
\node[arn_n] at (1,0) {}; 
\node[label=above:{$w_1$}] at (1,0) {};
\node[arn_n] at (2,0) {}; 
\node[label=above:{$w_2$}] at (2,0) {};
\node[arn_n] at (3,0) {}; 
\node[label=above:{$w_3$}] at (3,0) {};
\end{tikzpicture}} 
\ar[dl]_{w_0+w_1+w_2+w_3\choose w_0} 
\ar[d]|{{w_0+w_1+w_2+w_3\choose w_0+w_1}} 
\ar[dr]^{w_0+w_1+w_2+w_3\choose w_0+w_1+w_2} & \\
{\begin{tikzpicture}[scale=0.8]
\draw [thick] (1,0)--(2,0)--(3,0);
\node[arn_n] at (0,0) {};
\node[arn_n] at (1,0) {}; 
\node[arn_n] at (2,0) {}; 
\node[arn_n] at (3,0) {}; 
\end{tikzpicture}} 
\ar[d]|{-{w_1+w_2+w_3\choose w_1}} \ar[dr]|(.3){-{w_1+w_2+w_3\choose w_1+w_2}}  
& 
{\begin{tikzpicture}[scale=0.8]
\draw [thick] (0,0)--(1,0);
\draw [thick] (2,0)--(3,0);
\node[arn_n] at (0,0) {};
\node[arn_n] at (1,0) {}; 
\node[arn_n] at (2,0) {}; 
\node[arn_n] at (3,0) {}; 
\end{tikzpicture}} 
\ar[dl]^(.3){{w_0+w_1\choose w_0}} \ar[dr]_(.3){-{w_2+w_3\choose w_2}}  
&
{\begin{tikzpicture}[scale=0.8]
\draw [thick] (0,0)--(1,0)--(2,0);
\node[arn_n] at (0,0) {};
\node[arn_n] at (1,0) {}; 
\node[arn_n] at (2,0) {}; 
\node[arn_n] at (3,0) {}; 
\end{tikzpicture}} 
\ar[d]|{{w_0+w_1+w_2\choose w_0+w_1}} \ar[dl]|(.3){{w_0+w_1+w_2\choose w_0}}  
\\
{\begin{tikzpicture}[scale=0.8]
\draw [thick] (2,0)--(3,0);
\node[arn_n] at (0,0) {};
\node[arn_n] at (1,0) {}; 
\node[arn_n] at (2,0) {}; 
\node[arn_n] at (3,0) {}; 
\end{tikzpicture}} 
\ar[dr]_{w_2+w_3\choose w_2}
& 
{\begin{tikzpicture}[scale=0.8]
\draw [thick] (1,0)--(2,0);
\node[arn_n] at (0,0) {};
\node[arn_n] at (1,0) {}; 
\node[arn_n] at (2,0) {}; 
\node[arn_n] at (3,0) {}; 
\end{tikzpicture}} 
\ar[d]|{-{w_1+w_2\choose w_1}}
&
{\begin{tikzpicture}[scale=0.8]
\draw [thick] (0,0)--(1,0);
\node[arn_n] at (0,0) {};
\node[arn_n] at (1,0) {}; 
\node[arn_n] at (2,0) {}; 
\node[arn_n] at (3,0) {}; 
\end{tikzpicture}} 
\ar[dl]^{w_0+w_1\choose w_0}
\\
&
{\begin{tikzpicture}[scale=0.8]
\node[arn_n] at (0,0) {};
\node[arn_n] at (1,0) {}; 
\node[arn_n] at (2,0) {}; 
\node[arn_n] at (3,0) {}; 
\end{tikzpicture}}  & \\
}
\]
\caption{The complex $C_{\bullet}(w_0,w_1,w_2,w_3)$, $w_0,w_1,w_2\geq 0$, $w_3\in\bb{Z}$.}
\label{fig:Cw-complex}
\end{figure}

\begin{theorem}\label{thm:stab-coh-ribbon=hom-Cw}
    With notation as in Theorem~\ref{thm:res-Schur-ribbons}, and $N=|\ul{w}|$, we have
    \[H^j_{st}\left(\bb{S}_{\ll/\mu}\Omega\right) = H_{N-j}\left(C_{\bullet}(\ul{w})\right)\quad\text{ for all }j.\]
\end{theorem}

\begin{proof}
    From the resolution $\mf{F}_{\bullet}\Omega$ of $\bb{S}_{\ll/\mu}\Omega$ we get the hypercohomology spectral sequence
    \[E_1^{-i,j} = H^j\left(\PP,\mf{F}_i\Omega\right) \Longrightarrow \bb{H}^{j-i}\left(\PP,\mf{F}_{\bullet}\Omega\right) = H^{j-i}_{st}\left(\bb{S}_{\ll/\mu}\Omega\right). \]
    It follows from \eqref{eq:Hi-FJOmega} that $E_1^{-i,j}$ is non-zero only for $j=N$, and by construction we have $E_1^{-\bullet,N}=C_{\bullet}(\ul{w})$. It follows that the spectral sequence degenerates at the $E_2$-page, where we have
    \[ E_2^{-i,N} = H_i(C_{\bullet}),\quad E_2^{-i,j}=0\text{ for }j\neq N,\]
    from which the desired conclusion follows.
\end{proof}

\begin{remark}\label{rem:negative-wd}
    Even though for the construction of the complex $\mf{F}_{\bullet}^{\ul{w}}$ we assumed $w_i\geq 0$, it is clear from \eqref{eq:bin-coe-delJJ'} that we can allow $w_d<0$, and use the standard convention
    \begin{equation}\label{eq:x-choose-i}{x\choose i} = \frac{x(x-1)\cdots(x-i+1)}{i!}\quad\text{ for all }x\in\bb{Z}\text{ and all }i\geq 0.
    \end{equation}
    This generality will be important in Section~\ref{subsec:two-cols} (see also Section~\ref{subsec:duality-hook-2column}).
\end{remark}

\subsection{Two column partitions}\label{subsec:two-cols}

In this section we consider a partition $\ll$ whose Young diagram has two columns. We write $\ll'=(m,d)$ for the conjugate partition, and recall from \cite{AB1}*{Section~4} the construction of a resolution of the Schur functor $\bb{S}_{\ll}$. As remarked in \cite{AB1}*{Equations (5)-(6) in Section~1}, it is not possible to write a universal resolution by tensor products of exterior powers whose terms match up exactly with those in the Jacobi-Trudi determinantal expansion of the Schur function $s_{\ll}$. Nevertheless, a larger resolution exists and is constructed in \cite{AB1}*{Section~4}. We recall the construction here in a form that will be most convenient for our applications.

We define a complex $\mf{G}_{\bullet}=\mf{G}^{(m,d)}_{\bullet}$ in the category of polynomial functors, whose terms are
\[\mf{G}_t = \bigoplus_{D\in{[d]\choose t}} \mf{G}_{D},\quad\text{ where }\quad\mf{G}_D = \bw^{m+d_t} \oo \bw^{d-d_t}\quad\text{ for }D=\{d_1<\cdots<d_t\}.\]
To define the differential $\delta_t$, we decompose $\delta_t = \bigoplus\delta_{D,D'}$, where $\delta_{D,D'}:\mf{G}_D\lra \mf{G}_{D'}$. For $D=\{d_1<\cdots<d_t\}$ and $D' = D\setminus\{d_j\}$ we let $\op{id}$ denote the identity map, and let $\square$ denote the composition
\[\bw^{m+d_t} \oo \bw^{d-d_t} \overset{\Delta \oo \op{id}}{\lra} 
\left(\bw^{m+d_{t-1}} \oo \bw^{d_t-d_{t-1}}\right) \oo \bw^{d-d_t} 
=
\bw^{m+d_{t-1}} \oo \left(\bw^{d_t-d_{t-1}} \oo \bw^{d-d_t}\right) 
\overset{\op{id}\oo \wedge}{\lra} \bw^{m+d_{t-1}} \oo \bw^{d-d_{t-1}}\]
where $\Delta$ denotes comultiplication, and $\wedge$ denotes exterior multiplication. We then set $d_0=0$ and define
\begin{equation}\label{eq:del-DD'}
\delta_{D,D'} = \begin{cases}
    (-1)^{j+1}\cdot {d_{j+1}-d_{j-1} \choose d_j-d_{j-1}}\cdot\op{id} & \text{if }j<t; \\
    (-1)^{t+1+d_t - d_{t-1}}\cdot\square & \text{if }j=t. \\
\end{cases}
\end{equation}
With the notation above, the following is a consequence of \cite{AB1}*{Theorem~(12)}.

\begin{theorem}\label{thm:res-Schur-twocols}
 If $\ll'=(m,d)$ then the complex $\mf{G}_{\bullet}^{(m,d)}$ is a resolution of the Schur functor $\bb{S}_{\ll}$.
\end{theorem}

In analogy with the case of ribbon Schur functors, we get a resolution $\mf{G}_{\bullet}\Omega=\mf{G}_{\bullet}^{(m,d)}\Omega$ of $\bb{S}_{\ll}\Omega$ for $\ll=(m,d)'$, and study the associated hypercohomology spectral sequence. From Theorem~\ref{thm:tens-Omega=shift-coh}, we get that
\[ H^i(\PP,\mf{G}_D(\Omega)) = \begin{cases} \kk & \text{if }i=m+d, \\
0 & \text{otherwise},
\end{cases}\]
and define complexes $G_{\bullet}=G_{\bullet}(m,d)$ via
\[ G_{\bullet} = H^{m+d}\left(\PP,\mf{G}_\bullet(\Omega)\right).\]
We write $G_D = H^{m+d}(\PP,\mf{G}_D(\Omega))=\kk\cdot g_D$, and using \eqref{eq:del-DD'} and Corollary~\ref{cor:multOmega-iso-cohom}, we get that the differential in $G_{\bullet}$ is determined~by
\begin{equation}\label{eq:bin-coe-delDD'}
\delta_{D,D'}(g_D) = \begin{cases}
    (-1)^{j+1}\cdot {d_{j+1}-d_{j-1} \choose d_j-d_{j-1}}\cdot g_{D'} & \text{if }j<t; \\
    (-1)^{t+1+d_t-d_{t-1}}{m+d_{t} \choose d_t-d_{t-1}}\cdot g_{D'} & \text{if }j=t. \\
\end{cases}
\end{equation}
To establish the connection to the earlier complexes $C_{\bullet}$ we associate to $g_D$ the graph $\Gamma_{[d]\setminus D}$ with $d+1$ vertices having weights $(1^d,-m-d-1)$. We illustrate the case $d=3$ in Figure~\ref{fig:Gmd-complex}, where the graph $\Gamma_{[d]\setminus D}$ takes the place of the component $G_D$. Moreover, in writing the differential $\delta_{D,D'}$ we use the identity
\[(-1)^{t+1+d_t-d_{t-1}}{m+d_{t} \choose d_t-d_{t-1}} = (-1)^{t+1}\cdot {-m-1-d_{t-1}\choose d_t-d_{t-1}}.\]
The reader can check by comparing with Figure~\ref{fig:Cw-complex} that $G_{\bullet}(m,3)$ is dual to $C_{\bullet}(1,1,1,-m-4)$. We prove a similar statement for all $d$, as follows.

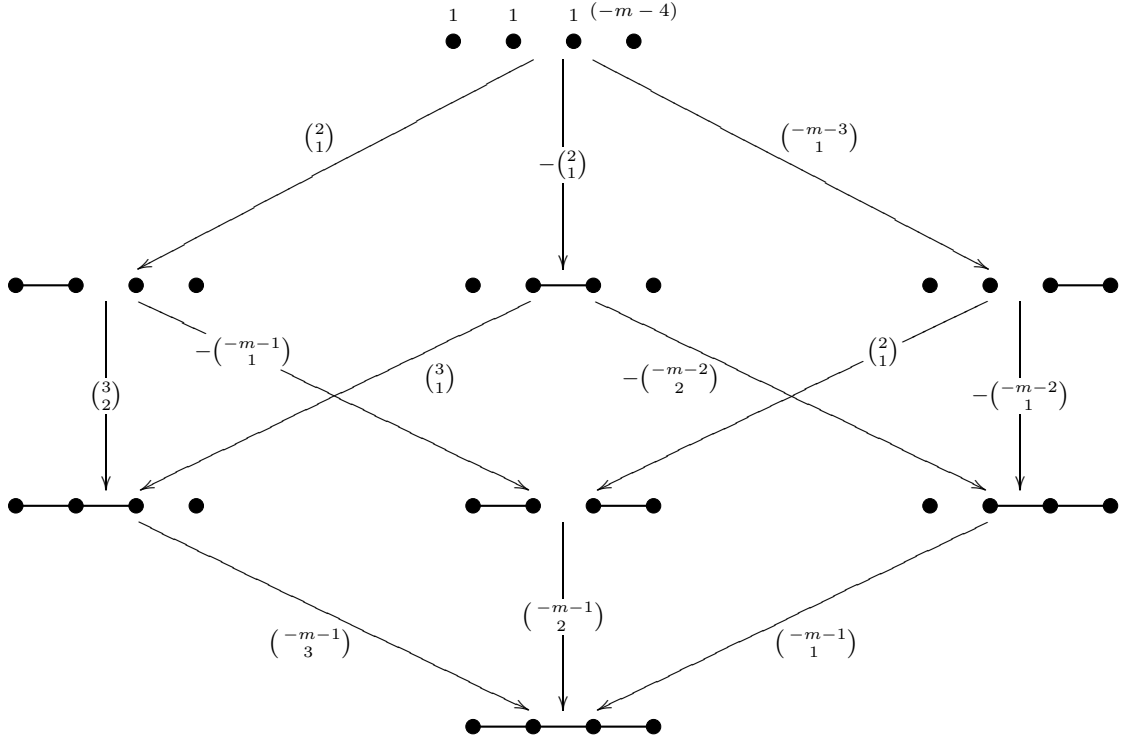
\begin{figure}[h]
\[
\xymatrix @C=7.5em @R=6.5em{
 & {\begin{tikzpicture}[scale=0.8]
\node[arn_n] at (0,0) {};
\node[label=above:{\tiny$1$}] at (0,0) {};
\node[arn_n] at (1,0) {}; 
\node[label=above:{\tiny$1$}] at (1,0) {};
\node[arn_n] at (2,0) {}; 
\node[label=above:{\tiny$1$}] at (2,0) {};
\node[arn_n] at (3,0) {}; 
\node[label=above:{\tiny$(-m-4)$}] at (3,0) {};
\end{tikzpicture}} 
\ar[dl]_{{2\choose 1}} 
\ar[d]|{-{2\choose 1}} 
\ar[dr]^{{-m-3\choose 1}} & \\
{\begin{tikzpicture}[scale=0.8]
\draw [thick] (0,0)--(1,0);
\node[arn_n] at (0,0) {};
\node[arn_n] at (1,0) {}; 
\node[arn_n] at (2,0) {}; 
\node[arn_n] at (3,0) {}; 
\end{tikzpicture}} 
\ar[d]|{{3\choose 2}} \ar[dr]|(.3){-{-m-1\choose 1}}  
& 
{\begin{tikzpicture}[scale=0.8]
\draw [thick] (1,0)--(2,0);
\node[arn_n] at (0,0) {};
\node[arn_n] at (1,0) {}; 
\node[arn_n] at (2,0) {}; 
\node[arn_n] at (3,0) {}; 
\end{tikzpicture}} 
\ar[dl]^(.3){{3\choose 1}} \ar[dr]_(.3){-{-m-2\choose 2}}  
&
{\begin{tikzpicture}[scale=0.8]
\draw [thick] (2,0)--(3,0);
\node[arn_n] at (0,0) {};
\node[arn_n] at (1,0) {}; 
\node[arn_n] at (2,0) {}; 
\node[arn_n] at (3,0) {}; 
\end{tikzpicture}} 
\ar[d]|{-{-m-2\choose 1}} \ar[dl]|(.3){{2\choose 1}}  
\\
{\begin{tikzpicture}[scale=0.8]
\draw [thick] (0,0)--(1,0)--(2,0);
\node[arn_n] at (0,0) {};
\node[arn_n] at (1,0) {}; 
\node[arn_n] at (2,0) {}; 
\node[arn_n] at (3,0) {}; 
\end{tikzpicture}} 
\ar[dr]_{-m-1\choose 3}
& 
{\begin{tikzpicture}[scale=0.8]
\draw [thick] (0,0)--(1,0);
\draw [thick] (2,0)--(3,0);
\node[arn_n] at (0,0) {};
\node[arn_n] at (1,0) {}; 
\node[arn_n] at (2,0) {}; 
\node[arn_n] at (3,0) {}; 
\end{tikzpicture}} 
\ar[d]|{-m-1\choose 2}
&
{\begin{tikzpicture}[scale=0.8]
\draw [thick] (1,0)--(2,0)--(3,0);
\node[arn_n] at (0,0) {};
\node[arn_n] at (1,0) {}; 
\node[arn_n] at (2,0) {}; 
\node[arn_n] at (3,0) {}; 
\end{tikzpicture}} 
\ar[dl]^{-m-1\choose 1}
\\
&
{\begin{tikzpicture}[scale=0.8]
\draw [thick] (0,0)--(1,0)--(2,0)--(3,0);
\node[arn_n] at (0,0) {};
\node[arn_n] at (1,0) {}; 
\node[arn_n] at (2,0) {}; 
\node[arn_n] at (3,0) {}; 
\end{tikzpicture}}  & \\
}
\]
\caption{The complex $G_{\bullet}(m,3)$, also the dual of $C_{\bullet}(1,1,1,-m-4)$}
\label{fig:Gmd-complex}
\end{figure}

\begin{theorem}\label{thm:two-columns}
    \begin{enumerate}
        \item For all $d\geq 0$ we have an isomorphism $G_{\bullet}(m,d) \simeq C_{\bullet}^{\vee}(1^d,-m-d-1)[-d]$.
        \item If $\ll'=(m,d)$ then the stable cohomology of $\bb{S}_{\ll}\Omega$ is given by
        \[ H^j_{st}\left(\bb{S}_{\ll}\Omega\right) = H_{j-m}\left(C_{\bullet}(1^d,-m-d-1)\right)\quad\text{ for all }j.\]
    \end{enumerate}
\end{theorem}

\begin{proof}
 To prove (1), we write $C_{\bullet}=C_{\bullet}(1^d,-m-d-1)$, and write $f_J^{\vee}\in C_J^{\vee}$ for the basis element dual to $f_J\in C_J$. We define a map $\Psi:G_{\bullet}(m,d) \lra C_{\bullet}^{\vee}(1^d,-m-d-1)[-d]$ by sending
 \[ \Psi(g_D) = f^{\vee}_{[d]\setminus D}\quad\text{ for all }D\subseteq [d].\]
 It is clear that $\Psi$ preserves homological degree, and is an isomorphism of the underlying vector spaces, so we only need to check that it commutes with the differentials. To that end, let $D=\{d_1<\cdots<d_t\}$ and $D'=D\setminus\{d_j\}$, and let $J=[d]\setminus D'$, $J'=[d]\setminus D = J\setminus\{d_j\}$. Notice that the edge $d_j$ in $\Gamma_J$ connects an interval of weight $d_j-d_{j-1}$ (to its left) to one of weight $d_{j+1}-d_j$ (to its right), where we set $d_{t+1}=-m-1$. Notice also that the number of missing edges in $\Gamma_J$ to the left of $d_j$ is $(j-1)$. It follows from \eqref{eq:bin-coe-delJJ'} that $\pd_{J,J'}$ is given as multiplication by the scalar
 \[(-1)^{j-1}\cdot {d_{j+1}-d_{j-1}\choose d_j-d_{j-1}},\]
 which agrees with the scalar in \eqref{eq:bin-coe-delDD'}, as desired.

 For (2), we argue as in the proof of Theorem~\ref{thm:stab-coh-ribbon=hom-Cw}. The hypercohomology spectral sequence
    \[E_1^{-i,j} = H^j\left(\PP,\mf{G}^{(m,d)}_i\Omega\right) \Longrightarrow H^{j-i}_{st}\left(\bb{S}_{\ll}\Omega\right) \]
satisfies $E_1^{-i,j}=0$ for $j\neq m+d$, and $E_1^{-\bullet,m+d} = G_{\bullet}(m,d)$. It follows that
\[H^j_{st}\left(\bb{S}_{\ll}\Omega\right) = H_{m+d-j}\left(G_{\bullet}(m,d)\right),\]
which by part (1) is isomorphic to (the dual of) $H_{j-m}\left(C_{\bullet}(1^d,-m-d-1)\right)$, concluding our proof.
\end{proof}

\section{Homology of some arithmetic complexes}\label{sec:arithmetic}

The goal of this section is to give a more in-depth study of the stable cohomology of Schur functors $\bb{S}_{\ll/\mu}\Omega$ for $\ll/\mu$ a ribbon, and for $\bb{S}_{\ll}\Omega$ when $\ll$ is a partition with two columns ($\ll_1=2$). Based on the results in Section~\ref{sec:resolutions-ext-pows}, this amounts to a thorough understanding of the homology of complexes of the form $C_{\bullet}(\ul{w})$ for certain tuples $\ul{w}$. We begin with a discussion of general weights $\ul{w}$, and then specialize our setting to the case when most $w_i$ are $1$, where we can give the most complete description of homology. We always assume that $w_0,\cdots,w_{d-1}\geq 0$, and $w_d\in\bb{Z}$ (possibly negative).

\subsection{General properties of the complexes $C_{\bullet}(\ul{w})$}

\begin{lemma}\label{lem:ses-C-complexes}
 For each $i=1,\cdots,d-1$, there exists a short exact sequence of complexes
 \[ 0 \lra C_{\bullet}(w_0,\cdots,w_{i-1}) \oo C_{\bullet}(w_i,\cdots,w_d) \overset{\iota}{\lra} C_{\bullet}(w_0,\cdots,w_d) \overset{\pi}{\lra} C_{\bullet-1}(w_0,\cdots,w_{i-1}+w_i,\cdots,w_d)\lra 0 \]
\end{lemma}

\begin{proof} Based on Figure~\ref{fig:Cw-complex}, one can think of $C_{\bullet}(w_0,\cdots,w_{i-1}) \oo C_{\bullet}(w_i,\cdots,w_d)$ as the subcomplex of $C_{\bullet}(\ul{w})$ corresponding to the graphs $\Gamma_J$ where the edge labelled by $i$ is missing ($i\not\in J$). Similarly, the terms of the complex $C_{\bullet-1}(w_0,\cdots,w_{i-1}+w_i,\cdots,w_d)$ correspond to graphs $\Gamma_J$ with $i\in J$: we then think of the nodes of weights $w_{i-1},w_i$ together with the connecting edge as a single node of weight $w_{i-1}+w_i$.

To be precise, if $J_1\subset[i-1]$ and $J_2\subset[d-i]$, and if we write $J_2+i = \{j+i: j\in J_2\}$, then 
\[ \iota(f_{J_1} \oo f_{J_2}) = f_J,\quad\text{where }J = J_1 \cup (J_2+i).\]
If $J\subset[d]$ with $i\in J$, we write 
\begin{equation}\label{eq:eps-J}
\epsilon(J) = \{j| j\in J,j<i\} \cup \{j-1 | j\in J,j>i\},
\end{equation}
and define
\[ \pi(f_J) = 0\quad\text{ if }i\not\in J,\text{ and }\quad \pi(f_J) = -f_{\epsilon(J)}\quad\text{ if }i\in J.\]
It is clear that in each homological degree we get an exact sequence of vector spaces, and we leave it to the reader to check that $\iota$ and $\pi$ are indeed maps of complexes.
\end{proof}

It will be important to understand more explicitly the connecting homomorphisms
\begin{equation}\label{eq:deltak-ses-Cw}
    \delta_k:H_k(C_{\bullet-1}(w_0,\cdots,w_{i-1}+w_i,\cdots,w_d)) \lra H_{k-1}(C_{\bullet}(w_0,\cdots,w_{i-1}) \oo C_{\bullet}(w_i,\cdots,w_d))
\end{equation}
in the long exact sequence in homology associated to the short exact sequence in Lemma~\ref{lem:ses-C-complexes}.

\begin{lemma}\label{lem:connecting-map-Cw}
    There is a morphism of complexes
    \[ \psi : C_{\bullet}(w_0,\cdots,w_{i-1}+w_i,\cdots,w_d) \lra C_{\bullet}(w_0,\cdots,w_{i-1}) \oo C_{\bullet}(w_i,\cdots,w_d),\]
    such that the induced maps in homology $H_{k-1}(\psi)$ agree with the connecting homomorphisms $\delta_k$ in \eqref{eq:deltak-ses-Cw}. More precisely, suppose $J'\subset[d-1]$ and let $t = |J'\cap[i-1]|$ denote the number of edges in $\Gamma_{J'}$ to the left of the vertex $v$ of weight $w_{i-1}+w_i$. Suppose also that $l\leq i-1$ and $r\geq i$ are such that the weight of the connected component of $\Gamma_{J'}$ containing $v$ is 
    \[w_l+\cdots+w_{i-1}+w_i+\cdots + w_r.\]
    If $J\subset[d]$ is the unique subset with $i\in J$ and $\epsilon(J)=J'$, and if we write $J\setminus\{i\}=J_1 \cup (J_2+i)$ for $J_1\subset[i-1]$ and $J_2\subset[d-i]$, then
    \[ \psi(f_{J'}) = (-1)^{t-1} {w_l+\cdots+w_{i-1}+w_i+\cdots + w_r \choose w_l+\cdots+w_{i-1}} \cdot f_{J_1}\oo f_{J_2}.\]
\end{lemma}

\begin{proof} We define a section $s$ of the map $\pi$ in Lemma~\ref{lem:ses-C-complexes} by letting $s(f_{J'})=-f_J$, where $i\in J\subset[d]$ is unique with the property $\epsilon(J)=J'$. To construct the connecting maps \eqref{eq:deltak-ses-Cw}, we take any cycle $z$ in $C_{k-1}(w_0,\cdots,w_{i-1}+w_i,\cdots,w_d)$, we lift it to $s(z)\in C_{k}(w_0,\cdots,w_{i-1},w_i,\cdots,w_d)$ and apply the differential to get $\delta(s(z))\in C_{k-1}(w_0,\cdots,w_{i-1},w_i,\cdots,w_d)$. We have $\pi(\delta(s(z)))=0$, hence $\delta(s(z))=\iota(z')$ for some $z'$. The map $\delta_k$ sends the homology class of $z$ to that of $z'$. Tracing through the construction, we see that this agrees with the map in homology induced by $\psi$.
\end{proof}

We define $P(\ul{w})$ to be the generating function for the homology of $C_{\bullet}(\ul{w})$:
\[P(\ul{w}) = \sum_{i\geq 0} h_i(C_{\bullet}(\ul{w}))\cdot t^i.\]

\begin{proposition}\label{prop:props-PCw}
    \begin{enumerate}
        \item If $\ll/\mu$ is the ribbon corresponding to the composition $\ul{w}$ and if $|\ul{w}|=N$ then
        \[ P(\ul{w}) = \sum_{i=0}^d h^{N-i}_{st}(\bb{S}_{\ll/\mu}\Omega)\cdot t^i.\]
        \item If $P(w_0,\cdots,w_{i-1})=0$ or $P(w_i,\cdots,w_d)=0$ then 
        \[P(\ul{w}) = t\cdot P(w_0,\cdots,w_{i-1}+w_i,\cdots,w_d).\]
        \item If $\chr(\kk)=p$ then $P(\ul{w}) = P(p\cdot\ul{w})$.
        \item If $\chr(\kk)=p$ and $p^k>w_0+\cdots+w_{d-1}$ then $P(\ul{w}) = P(w_0,\cdots,w_{d-1},w_d+p^k)$.
    \end{enumerate}
\end{proposition}

\begin{proof}
    Part (1) follows from Theorem~\ref{thm:stab-coh-ribbon=hom-Cw}. For part (2), the hypothesis implies that one of the complexes $C_{\bullet}(w_0,\cdots,w_{i-1})$ or $C_{\bullet}(w_i,\cdots,w_d)$ is exact, hence the same is true for $C_{\bullet}(w_0,\cdots,w_{i-1}) \oo C_{\bullet}(w_i,\cdots,w_d)$. It follows that the map $\pi$ in Lemma~\ref{lem:ses-C-complexes} is a quasi-isomorphism, which yields the desired conclusion.

    Parts (3), (4) are familiar consequences of Lucas' theorem \cite{lucas}*{Chapter XXIII, Section 228}. Since this is usually stated for non-negative entries in the binomial coefficients, we include a short proof. With notation as in \eqref{eq:x-choose-i}, we have for all $m\in\bb{Z}$ the power series identity
    \[(1+x)^m = \sum_{k\geq 0} {m\choose k}\cdot x^k \in \bb{Z}[[x]].\]
    Using the universal map $\bb{Z}\lra\kk$, and the identity $(1+x)^{pm}=(1+x^p)^m$ in $\kk[[x]]$, we get by considering the coefficient of $x^{pv}$ that 
    \[ {pm\choose pv} = {m\choose v} \quad\text{in }\kk\text{ for all }m\in\bb{Z},v\geq 0.\]
    This implies that $C_{\bullet}(\ul{w})=C_{\bullet}(p\cdot\ul{w})$ are equal as complexes of $\kk$-vector spaces, which proves (3). Using
    \[ (1+x)^{m+p^k} = (1+x)^m\cdot (1+x)^{p^k} = (1+x)^m\cdot (1+x^{p^k}) \quad\text{in }\kk[[x]],\]
    it follows by considering the coefficient of $x^v$ that ${m\choose v}={m+p^k\choose v}$ whenever $p^k>v$. This implies the equality of complexes of $\kk$-vector spaces $C_{\bullet}(\ul{w})=C_{\bullet}(w_0,\cdots,w_{d-1},w_d+p^k)$, proving~(4).
\end{proof}

\begin{corollary}\label{cor:vanish-hooks+group-p}
    \begin{enumerate}
        \item If $b\geq 1$ and $j\not\equiv 0\ (\textrm{mod }p)$ then $P(pb,1^j)=0$.
        \item For all $a\geq 0$, $b\geq 1$ we have
        \[P(pb,1^{pa}) = t^{a(p-1)} \cdot P(pb,p^a) = t^{a(p-1)} \cdot P(b,1^a).\]
    \end{enumerate}
\end{corollary}

\begin{proof} To prove part (1), we note that the ribbon corresponding to the composition $\ul{w}=(pb,1^j)$ is the hook partition $\ll=(j+1,1^{pb-1})$. Its $p$-core is $\gamma=\op{core}_p(\ll)=(j'+1,p-1)$ where $1\leq j'<p$ is the remainder of division of $j$ by $p$. Since $\gamma_1\geq 2$, Theorem~\ref{thm:vanishing-p-core} proves $h_{st}^j(\bb{S}_{\ll}\Omega)=0$ for all $j$. We then get $P(pb,1^j)=0$ using Proposition~\ref{prop:props-PCw}(1).

    For part (2), the second equality follows from Proposition~\ref{prop:props-PCw}(3), so we focus on the first equality. Iterating Proposition~\ref{prop:props-PCw}(2) for $i=pa,pa-1,\cdots,pa-p+2$ and using $P(pb,1^{i-1})=0$, we get
    \[ P(pb,1^{pa}) = t\cdot P(pb,1^{pa-2},2) = t^2\cdot P(pb,1^{pa-3},3) = \cdots = t^{p-1}\cdot P(pb,1^{p(a-1)},p).\]
    Iterating this with $i=p(a-r),p(a-r)-1,\cdots,p(a-r-1)+2$ for each $r=1,\cdots,a-1$, we get 
    \[P(pb,1^{p(a-r)},p^r) = t\cdot P(pb,1^{p(a-r)-2},2,p^r) = t^2\cdot P(pb,1^{p(a-r)-3},3,p^r) = \cdots = t^{p-1}\cdot P(pb,1^{p(a-r-1)},p^{r+1}),\]
    which implies $P(pb,1^{pa}) = t^{(p-1)r}\cdot P(pb,1^{p(a-r)},p^r)$ for each $r=1,\cdots,a$, and in particular for $r=a$ we obtain the desired identity.
\end{proof}

A natural variation on the complexes $C_{\bullet}(\ul{w})$ comes by considering a fixed subset $A\subset [d]$ and defining $C_{\bullet}^A(\ul{w})$ to be the subcomplex spanned by those $f_J$ with $A\not\subseteq J$. In Figure~\ref{fig:Cw-complex}, this amounts to removing all the graphs that contain all the edges from $A$. It is clear from the definition of the differential in $C_{\bullet}(\ul{w})$ that each $C_{\bullet}^A(\ul{w})$ is a subcomplex. If we instead restrict to graphs containing a given edge set, then the resulting complexes are naturally identified with quotients of $C_{\bullet}(\ul{w})$. There are however circumstances under which they are also subcomplexes, as seen by the following.

\begin{proposition}\label{prop:Cpa-summand-C1pa}
    There is a natural inclusion of complexes $\varphi:C_{\bullet-a(p-1)}(pb,p^a)\lra C_{\bullet}(pb,1^{pa})$ which is a quasi-isomorphism. More precisely, if we let
    \[ A = \{j\in[pa] | j\not\equiv 1\ (\text{mod }p)\}\]
    and define for $J\subset[a]$ the set
    \[J^{(p)} = \{jp-p+1 | j\in J\} \cup A,\]
    then $\varphi$ is defined via $\varphi(f_J) = f_{J^{(p)}}$.
\end{proposition}

\begin{proof} We first note that $|J^{(p)}|=|J|+a\cdot(p-1)$, so the map $\varphi$ preserves homological degrees. We need to check that $\varphi$ commutes with the differential. We have for all $J\subset[a]$ that $\ul{w}(J)=\ul{w}(J^{(p)})$, that is, the weights of the connected components of $\Gamma_J$ and $\Gamma_{J^{(p)}}$ coincide, and in particular they are all divisible by $p$. If we write $j^{(p)}=jp-p+1$ for $j\in J$ then $(J\setminus\{j\})^{(p)} = J^{(p)}\setminus \{j^{(p)}\}$ and it is clear that $\pd_{J,J\setminus\{j\}} = \pd_{J^{(p)},J^{(p)}\setminus \{j^{(p)}\}}$ for all $J\subset[a]$ and all $j\in J$. To conclude that $\varphi$ is a map of complexes, it remains to check that $\pd_{J^{(p)},J'}=0$ whenever $J'$ is not of the form $J^{(p)}\setminus\{j^{(p)}\}$. This follows from the fact that the removal of any edge $j'\neq j^{(p)}$ from $J^{(p)}$ will break a component of weight $w^{J^{(p)}}_t \equiv 0\ (\textrm{mod }p)$ into two components of weights $w^{J'}_t$ and $w^{J'}_{t+1}$ which are not divisible by $p$. This implies that the binomial coefficient in \eqref{eq:bin-coe-delJJ'} vanishes as desired.

By construction, we see that the image of $\varphi$ is in fact a direct summand of $C_{\bullet}(pb,1^{pa})$: indeed, we have a decomposition
\[C_{\bullet}(pb,1^{pa}) = \op{Im}(\varphi)\oplus C^A_{\bullet}(pb,1^{pa}),\]
and in particular $\varphi$ induces an injection in homology. To show that it is a quasi-isomorphism, it suffices to prove that we have an abstract isomorphism 
\[H_{i-a(p-1)}\left(C_{\bullet}(pb,p^a)\right)\simeq H_{i}\left(C_{\bullet}(pb,1^{pa})\right),\]
which follows from Corollary~\ref{cor:vanish-hooks+group-p}(2).
\end{proof}

\subsection{Recursion for hook partitions}\label{subsec:recursion-hooks}

We recall \eqref{eq:Hab} and note that in light of \eqref{eq:Flag-to-Pspace} we have an alternative description for the polynomials $H_{a,b}=H_{a,b}(t)$:
\begin{equation}\label{eq:def-Hab}
    H_{a,b}(t)=\sum_{i\geq 0}h^i_{st}(\bb{S}_{(a,1^b)}\Omega)\cdot t^i.
\end{equation}
Our next goal is to explain a simple recursion satisfied by $H_{a,b}$.

\begin{theorem}\label{thm:recursion-hooks}
    The polynomials $H_{a,b}=H_{a,b}(t)$ satisfy the following relations:
    \begin{enumerate}
        \item $H_{0,0}=1$ and $H_{1,b}=t^{b+1}$ for all $b\geq 0$.
        \item If $p\nmid a+b$ then
        \[H_{a,b} = \begin{cases}
            t^b\cdot H_{a,0} & \text{if }a\equiv 1\ (\op{mod}\ p); \\
            0 & \text{otherwise}.
        \end{cases}\]
        \item $H_{a,0}=t\cdot H_{a-1,0}$ if $a\equiv 1\ (\op{mod}\ p)$.
        \item $H_{pa-i,pb+i} = t^i\cdot H_{pa,pb}$ for $i=1,\cdots,p-2$.
        \item $H_{pa+1,pb-1}=t^{pb-b}\cdot H_{a+1,b-1}$.
        \item $H_{pa,pb}= t^{pb+1}\cdot H_{pa-p+1,0}+t^{pb-b}\cdot H_{a,b}$.
    \end{enumerate}
\end{theorem}

\begin{proof}
 {\bf (1).} If $a=b=0$ then $\bb{S}_{(a,1^b)}\Omega=\mc{O}_{\PP}$ whose only cohomology is $H^0(\PP,\mc{O}_{\PP})=\kk$. Likewise, if $a=1$ then $\bb{S}_{(a,1^b)}\Omega=\Omega^{b+1}$, whose cohomology is described by \eqref{eq:Hd-of-Omegad}. 

 {\bf (2).} If $\ll=(a,1^b)$ and $p\nmid a+b$ then $\op{core}_p(\ll)=(a',1^{b'})$, where $a'=p$ if $p|a$, and $1\leq a'<p$ is the remainder upon division of $a$ by $p$ when $p\nmid a$. If $a'\neq 1$ then $H_{a,b}=0$ by Theorem~\ref{thm:vanishing-p-core}. If $a'=1$ then we use the short exact sequence
 \[0\lra \bb{S}_{(a,1^b)}\Omega \lra \Sym^a\Omega \oo \bw^b \Omega \lra \bb{S}_{(a+1,1^{b-1})}\Omega \lra 0 \]
 and the fact that $H_{a+1,b-1}=0$ to conclude that the stable cohomology of $\bb{S}_{(a,1^b)}\Omega$ agrees with that of $\Sym^a\Omega \oo \bw^b \Omega$. Using the K\"unneth formula (Theorem~\ref{thm:Kunneth-coh-omega}) we conclude that $H_{a,b}=t^b\cdot H_{a,0}$.

 {\bf (3).} Using the short exact sequence
 \[ 0 \lra \bb{S}_{(a-1,1)}\Omega \lra \Sym^{a-1}\Omega \oo \Omega \lra \Sym^a\Omega \lra 0 \]
 and the fact that $H_{a-1,1}=0$ (established in (2)), it follows that $\Sym^{a-1}\Omega \oo \Omega$ and $\Sym^a\Omega$ have the same stable cohomology. K\"unneth's formula then shows that $H_{a,0}=t\cdot H_{a-1,0}$.

 {\bf (4).} It follows from (2) that $H_{pa-i,0}=0$ for $i=1,\cdots,p-2$. Combining this with the short exact sequence
 \[0\lra \bb{S}_{(pa-i,1^{pb+i})}\Omega \lra \Sym^{pa-i}\Omega \oo \bw^{pb+i} \Omega \lra \bb{S}_{(pa-i+1,1^{pb+i-1})}\Omega \lra 0\]
 and the K\"unneth formula, it follows from the long exact sequence in cohomology that $H_{pa-i,pb+i}=t\cdot H_{pa-i+1,pb+i-1}$ for $i=1,\cdots,p-2$, which implies (4). 

 {\bf (5).} Since the partition $\ll=(pa+1,1^{pb-1})$ is the ribbon corresponding to the composition $(pb,1^{pa})$, it follows from Proposition~\ref{prop:props-PCw}(1) and Corollary~\ref{cor:vanish-hooks+group-p}(2) that
 \[ 
 \begin{aligned}    
 H_{pa+1,pb-1} = t^{pa+pb}\cdot P(pb,1^{pa})(t^{-1}) &= t^{a+pb} \cdot P(b,1^a)(t^{-1}) \\
 &= t^{a+pb} \cdot t^{-a-b} \cdot H_{a+1,b-1} = t^{pb-b}\cdot H_{a+1,b-1},
 \end{aligned}\]
which proves (5).

 {\bf (6).} We consider the short exact sequence in Lemma~\ref{lem:ses-C-complexes} with $i=1$, $d=pa+1$, and composition $\ul{w}$ given by $w_0=pb-1$ and $w_1=\cdots=w_d=1$. We have 
 \[C_{\bullet}(w_0,\cdots,w_{i-1}) \oo C_{\bullet}(w_i,\cdots,w_d) = C_{\bullet}(w_0) \oo C_{\bullet}(w_1,\cdots,w_d) \simeq C_{\bullet}(w_1,\cdots,w_d)=C_{\bullet}(1^{pa+1}),\]
 and 
 \[C_{\bullet}(w_0,\cdots,w_{i-1}+w_i,\cdots,w_d) = C_{\bullet}(pb,1^{pa}).\]
 We claim that the connecting homomorphisms in the associated long exact sequence are all zero. To prove that, we use the explicit description from Lemma~\ref{lem:connecting-map-Cw}, together with the fact that by Proposition~\ref{prop:Cpa-summand-C1pa}, every homology class in $C_{\bullet}(pb,1^{pa})$ can be represented by a linear combination of basis elements $f_{J^{(p)}}$. With the notation in Lemma~\ref{lem:connecting-map-Cw}, it suffices to check that $\psi(f_{J^{(p)}})=0$. This follows from the fact that the binomial coefficient appearing in the definition of $\psi$ takes the form
 \[ {pb+pk \choose pb-1}\quad\text{ for some }k,\]
 which vanishes in characteristic $p$. We conclude that
 \[ P(pb-1,1^{pa+1}) = P(1^{pa+1}) + t\cdot P(pb,1^{pa}).\]
 Substituting $t\lra t^{-1}$ and multiplying with $t^{p(a+b)}$ we get
 \[ H_{pa+2,pb-2} = t^{pb-1}\cdot H_{pa+1,0} + t^{-1}\cdot H_{pa+1,pb-1}.\]
 Using parts (4) and (5) we get
 \[ t^{p-2}\cdot H_{p(a+1),p(b-1)} = t^{pb-1}\cdot H_{pa+1,0} + t^{pb-b-1}\cdot H_{a+1,b-1}.\]
 Replacing $a\mapsto a-1$ and $b\mapsto b+1$ and dividing by $t^{p-2}$ yields (6) and concludes our proof.
\end{proof}

\begin{example}
    Let $\lambda = (3,1^3)$ and consider computing the cohomology of $\bbs_\lambda\Omega$ in characteristic $2$. Writing $3 = 2 \cdot 1 + 1$ and $3 = 2 \cdot 2 -1$, we get using Theorem~\ref{thm:recursion-hooks}(5) for $p=2$, $a=1$, $b=2$, and Theorem~\ref{thm:recursion-hooks}(2) that
    $$H_{3,3} = t^2\cdot H_{2,1} = 0.$$
    Notice that the $2$-core of $\ll$ is empty, so the criterion for vanishing in Theorem~\ref{thm:vanishing-p-core} does not apply to $\ll$.
\end{example}

We can apply the recursive formulas in Theorem~\ref{thm:recursion-hooks} to prove the following strange symmetry, that will be used in Section~\ref{subsec:duality-hook-2column} to relate stable cohomology for hook and two-column partitions.

\begin{corollary}\label{cor:strange-symm-Hab}
 If $p^k\geq 2(A+B)$ then
 \[ H_{A,p^k-2A-B} = t^{p^k-2(A+B)}\cdot H_{A,B}.\]
\end{corollary}

\begin{proof} We argue by induction on $A$, noting that the case $A=1$ follows from Theorem~\ref{thm:recursion-hooks}(1). If $p\nmid A+B$, then we can use Theorem~\ref{thm:recursion-hooks}(2) to conclude the equality: if $A\equiv 1\ (\op{mod}\ p)$ then both sides equal $t^{p^k-2A-B}\cdot H_{A,0}$, and otherwise both sides vanish. We will therefore assume for the rest of the proof that $p|A+B$.

\textbf{Case 1:} $A=pa-i$, $B=pb+i$, with $a\geq 1$, $0\leq i\leq p-2$. Using Theorem~\ref{thm:recursion-hooks}(4), we may assume that $i=0$. It follows from Theorem~\ref{thm:recursion-hooks}(6) that
\[
\begin{aligned}
H_{A,p^k-2A-B} &= t^{p^k-2A-B+1}\cdot H_{A-p+1,0} + t^{p^k-2A-B-(p^{k-1}-2a-b)}\cdot H_{a,p^{k-1}-2a-b} \\
&=t^{p^k-2(A+B)}\cdot\left(t^{B+1}\cdot H_{A-p+1,0} + t^{B-b}\cdot H_{a,b}\right) ,
\end{aligned}\]
where the last equality uses the inductive hypothesis (which applies since $p^{k-1}\geq 2(A+B)/p = 2(a+b)$, and $a<A$). Theorem~\ref{thm:recursion-hooks}(6) implies $H_{A,B}=t^{B+1}\cdot H_{A-p+1,0} + t^{B-b}\cdot H_{a,b}$ and proves the desired identity.

\textbf{Case 2:} $A=pa+1$, $B=pb-1$ for some $b\geq 1$. It follows from Theorem~\ref{thm:recursion-hooks}(5) that
\begin{equation}\label{eq:HA-pk-2AB-first-appx}
H_{A,p^k-2A-B} = t^{p^k-2A-B+1-(p^{k-1}-2a-b)}\cdot H_{a+1,p^{k-1}-2a-b-1}.
\end{equation}
We can apply the induction hypothesis to $(a+1,b-1)$, since $a+1<A$ and $2(a+1+b-1)\leq p^{k-1}$, and get
\[H_{a+1,p^{k-1}-2a-b-1} = t^{p^{k-1}-2a-2b}\cdot H_{a+1,b-1}.\]
Plugging this into \eqref{eq:HA-pk-2AB-first-appx} yields
\[H_{A,p^k-2A-B} = t^{p^k-2(A+B)}\cdot\left(t^{B+1-b}\cdot H_{a+1,b-1}\right)\]
and the desired conclusion now follows from Theorem~\ref{thm:recursion-hooks}(5).
\end{proof}

\subsection{Symmetric powers of $\Omega$}\label{subsec:coh-Sym-Omega}

We can now prove \eqref{eq:stab-O(-d,d)}, which is equivalent via \eqref{eq:Flag-to-Pspace} to the following.

\begin{theorem}\label{thm:gen-fun-stab-coh-SymOmega}
    For $d\geq 0$ we have that
    \begin{equation}\label{eq:genfun-SymOmega} \sum_{i\geq 0}h^i_{st}(\Sym^d\Omega)\cdot t^i = \sum_{\ul{a}\in A_{p,d}} t^{|\ul{a}|_p}
    \end{equation}
\end{theorem}

\begin{proof}
    Since $\Sym^d\Omega = \bb{S}_{(a,1^b)}\Omega$ for $a=d$ and $b=0$, it follows that the left side of \eqref{eq:genfun-SymOmega} is $H_{d,0}$. We denote the right side of \eqref{eq:genfun-SymOmega} by $A(d)$, and prove by induction on $d$ that $H_{d,0}=A(d)$. When $d\not\equiv 0,1\ (\op{mod}\ p)$ we have that $A_{p,d}=\emptyset$ hence $A(d)=0$, while $H_{d,0}=0$ by Theorem~\ref{thm:recursion-hooks}(2).

    If $d\equiv 1\text{ (mod $p$)}$, then we have a bijection $\phi:A_{p,d-1}\lra A_{p,d}$ given by
    \[ \phi(a_0,a_1,a_2,\cdots) = (a_0+1,a_1,a_2,\cdots).\]
    Since $|\phi(\ul{a})|_p = |\ul{a}|_p+1$ for $\ul{a}\in A_{p,d-1}$, it follows that $A(d)=t\cdot A(d-1)$. By Theorem~\ref{thm:recursion-hooks}(3), we have $H_{d,0}=t\cdot H_{d-1,0}$, hence the equality $H_{d,0}=A(d)$ follows by induction.

    Suppose finally that $d\equiv 0\text{ (mod $p$)}$. By Theorem~\ref{thm:recursion-hooks}(6), we have
    \[ H_{d,0} = t\cdot H_{d-p+1,0} + H_{d/p,0},\]
    so it suffices to prove the analogous identity for $A(d)$. We have an injection $\iota:A_{p,d/p}\lra A_{p,d}$ defined by
    \[ \iota(a_0,a_1,a_2,\cdots) = (0,a_0,a_1,a_2,\cdots)\]
    which satisfies $|\iota(\ul{a})|_p = |\ul{a}|_p$. The complement $A_{p,d}\setminus\iota(A_{p,d/p})$ consists precisely of those $\ul{a}\in A_{p,d}$ with $a_0\neq 0$. Since $a_0\equiv d\equiv 0\text{ (mod $p$)}$, it follows that $a_0\geq p$ for all such $\ul{a}$. We get a bijection
    \[\psi: A_{p,d}\setminus\iota(A_{p,d/p}) \lra A_{p,d-p+1},\quad \psi(\ul{a})=(a_0-p+1,a_1,a_2,\cdots).\]
    Since $|\psi(\ul{a})|_p=|\ul{a}|_p-1$, it follows that
    \[A(d) = \sum_{\a\in A_{p,d/p}}t^{|\iota(\ul{a})|_p} + \sum_{\ul{b}\in A_{p,d-p+1}}t^{|\ul{b}|_p + 1} = A(d/p) + t\cdot A(d-p+1),\]
    concluding our proof.
\end{proof}

\begin{remark}
 As it was pointed to us by one of the referees, the reciprocal of the polynomial \eqref{eq:genfun-SymOmega} agrees with the Poincar\'e polynomial encoding the extension groups $\Ext^{\bullet}(\bw^d,D^d)$, or equivalently $\Ext^{\bullet}(\Sym^d,\bw^d)$. These groups were first computed by Akin \cite{akin-ext}, and the computations were further extended in work of Cha\l upnik \cite{chalup-Adv} and Touz\'e \cite{touze-bar-complexes}. In the subsequent work \cite{RV} we explore this connection in detail, and explain how it relates more broadly to Koszul--Ringel duality on the category of strict polynomial functors as described in the work of Cha\l upnik, Krause, and Touz\'e \cites{chalup-Adv,Krause-duality,touze}.
 \end{remark}

\begin{example} It follows from Theorem~\ref{thm:gen-fun-stab-coh-SymOmega} that the cohomological degrees where $H^k_{st}\left(\Sym^d\Omega\right)\neq 0$ are all in the range $1\leq k\leq d$. If $\chr(\kk)=2$ and $d=2^r$ then non-vanishing occurs for all such degrees:
$$H^k_{st}\left(\Sym^{2^r}\Omega\right) \neq 0 \quad \textrm{for all} \ k=1 , \dots , 2^r.$$
Indeed, consider $1\leq k\leq 2^r$ and write $k=2^{r-i}+j$ for some $1\leq i \leq r$ and $0\leq j \leq 2^{r-i}$. Consider the tuple
\[\ul{a} = (0 , \dots , 2j,2^{r-i} -j,0,\dots , 0),\text{ where }a_{i-1} = 2j\text{ and }a_i = 2^{r-i} - j,\]
and note that since $a_{i-1}\cdot 2^{i-1}+a_i\cdot 2^i = 2^r$, we have $\ul{a}\in A_{2,2^r}$. Since $|\ul{a}|_2=a_{i-1}+a_i=k$, it follows from \eqref{eq:genfun-SymOmega} that $H^k_{st}\left(\Sym^{2^r}\Omega\right) \neq 0$, as desired.
\end{example}

\begin{example}\label{ex:cohSyma-small-vals}
We record below the stable cohomology of small symmetric powers of $\Omega$ in low characteristics.
\begin{center}
\def\arraystretch{1.2}
\begin{tabular}{|c|c|c|c|c|}
\hline
& $\chr(\kk)=2$ & $\chr(\kk)=3$         & $\chr(\kk)=5$       & $\chr(\kk)=7$     \\ \hline
$H_{1,0}$ & $t$ & $t$ & $t$ & $t$ \\ \hline
$H_{2,0}$ & $t + t^2$                   & 0           & 0         & 0       \\ \hline
$H_{3,0}$                        & $t^2 + t^3$                 & $t + t^2$   & 0         & 0       \\ \hline
$H_{4,0}$                        & $t + t^2 + t^3 + t^4$       & $t^2 + t^3$ & 0         & 0       \\ \hline
$H_{5,0}$                        & $t^2+t^3+t^4+t^5$           & 0           & $t+t^2$   & 0       \\ \hline
$H_{6,0}$                        & $t+2t^2+t^3+t^4+t^5$        & $t^3+t^4$   & $t^2+t^3$ & 0       \\ \hline
$H_{7,0}$                        & $t^2+2t^3+t^4 + t^5+t^6$    & $t^4+t^5$   & 0         & $t+t^2$ \\ \hline
$H_{8,0}$ & $t+t^2+t^3+2t^4+2t^5+t^6+t^7+t^8$ & 0 & 0 & $t^2+t^3$ \\ \hline
\end{tabular}
\end{center}
\end{example}

\subsection{Generating functions and non-vanishing for hooks}\label{subsec:generating}

Using Theorem~\ref{thm:recursion-hooks} we can now prove \eqref{eq:Npb+p-1}, \eqref{eq:Npb+i}, and \eqref{eq:Hab-nonzero}. To that end, we define
\[ \mc{H}_b(t,u) = \sum_{a\geq 1}H_{a,b}(t)\cdot u^a,\]
so that $\mc{N}_b(t,u) = u^b/t^b\cdot \mc{H}_b(t,u)$, and let
\[ \mc{A}(t,u) = \sum_{d\geq 0}H_{pd,0}(t)\cdot u^{pd} = \sum_{d\geq 0} A(pd)\cdot u^{pd},\]
where the last equality uses the notation in the proof of Theorem~\ref{thm:gen-fun-stab-coh-SymOmega}. Using the fact that $A(pd+1)=t\cdot A(pd)$, and $A(pd+i)=0$ for $i=2,\cdots,p-1$, we have that
\[ \sum_{d\geq 0} A(d)\cdot u^d = \mc{A}(t,u)\cdot (1+tu).\]
The recursion for $A(d)=H_{d,0}$ yields then
\[
\begin{aligned}
    \mc{A}(t,u) &= \sum_{d\geq 0} \left(A(d)+t\cdot A(pd-p+1)\right)\cdot u^{pd} = \left(\sum_{d\geq 0} A(d)\cdot (u^p)^d\right) + \left(\sum_{d\geq 0}A(pd-p)\cdot t^2\cdot u^{pd} \right) \\
    &= \mc{A}(t,u^p)\cdot(1+tu^p) + \mc{A}(t,u)\cdot t^2\cdot u^p.
\end{aligned}
\]
We can rewrite the above equality as
\[ \mc{A}(t,u) = \frac{1+t\cdot u^p}{1-t^2\cdot u^p}\cdot\mc{A}(t,u^p),\]
which immediately yields \eqref{eq:def-Atu}. Referring to parts (1)--(6) of Theorem~\ref{thm:recursion-hooks} and writing $H_{a,b}=H_{a,b}(t)$, we obtain for $b\geq 1$
\[
\begin{aligned}
\mc{H}_{pb-1}(t,u) &= \sum_{a\geq 1} H_{a,pb-1}\cdot u^a \overset{(2)}{=} \sum_{a\geq 0} H_{pa+1,pb-1}\cdot u^{pa+1} \\
&\overset{(5)}{=} \sum_{a\geq 0} H_{a+1,b-1} \cdot t^{(p-1)b}\cdot u^{pa+1} = \sum_{a\geq 1} H_{a,b-1} \cdot t^{(p-1)b}\cdot u^{pa-p+1} = \frac{t^{(p-1)b}}{u^{p-1}}\cdot\mc{H}_{b-1}(t,u^p).
\end{aligned}
\]
It follows that
\[\mc{N}_{pb-1}(t,u) = \frac{u^{pb-1}}{t^{pb-1}}\cdot\mc{H}_{pb-1}(t,u) = \frac{u^{p(b-1)}}{t^{b-1}}\cdot\mc{H}_{b-1}(t,u^p) = \mc{N}_{b-1}(t,u^p),\]
which yields \eqref{eq:Npb+p-1} after substituting $b\mapsto b+1$.

To prove \eqref{eq:Npb+i}, we use again parts (1)--(6) of Theorem~\ref{thm:recursion-hooks}. We let $i=0,\cdots,p-2$ and have
\[ 
\begin{aligned}
    \mc{H}_{pb+i}(t,u) &= \sum_{a\geq 1}H_{a,pb+i}\cdot u^a \overset{(2)}{=} \left(\sum_{a\geq 1}H_{pa-i,pb+i}\cdot u^{pa-i} \right) + \left(\sum_{a\geq 0}H_{pa+1,pb+i}\cdot u^{pa+1} \right) \\
    &\overset{(4),(2)}{=}\left(\sum_{a\geq 1}H_{pa,pb}\cdot t^i\cdot u^{pa-i} \right) + t^{pb+i}\cdot\left(\sum_{a\geq 0}H_{pa+1,0}\cdot u^{pa+1}\right) \\
    &\overset{(6),(3)}{=}\left[\sum_{a\geq 1}\left(t^{pb+2}\cdot H_{pa-p,0}+t^{(p-1)b}\cdot H_{a,b}\right)\cdot t^i\cdot u^{pa-i} \right] + t^{pb+i}\cdot t\cdot u\cdot\mc{A}(t,u) \\
    &= t^{pb+i+2}\cdot u^{p-i}\cdot\left(\sum_{a\geq 1}H_{pa-p,0}\cdot u^{pa-p} \right) + \frac{t^{(p-1)b+i}}{u^i}\cdot\left(\sum_{a\geq 1}H_{a,b}\cdot u^{pa}\right) + t^{pb+i+1}\cdot u\cdot\mc{A}(t,u) \\
    &= \left(t^{pb+i+2}\cdot u^{p-i}+t^{pb+i+1}\cdot u\right)\cdot\mc{A}(t,u) + \frac{t^{(p-1)b+i}}{u^i}\cdot \mc{H}_b(t,u^p).
\end{aligned}
\]
Multiplying by $(u/t)^{pb+i}$ yields \eqref{eq:Npb+i}, as desired.

We are left with checking \eqref{eq:Hab-nonzero}. We write $b=-1+c\cdot q$, $p\nmid c$, and by iterating \eqref{eq:Npb+p-1} we get that
\[ \mc{N}_b(t,u) = \mc{N}_c(t,u^q),\]
so we can reduce the problem to the case $q=1$, or equivalently, to $b\not\equiv -1\ (\op{mod}\ p)$. Under this assumption, we have to show that
\[H_{a,b}(t)\neq 0 \Longleftrightarrow p|a-1\quad\text{ or }\quad p|a+b.\]
If $p\nmid a+b$ then the equivalence holds by Theorem~\ref{thm:recursion-hooks}(2), since $H_{a,0}(t)\neq 0$ whenever $a\equiv 1\ (\op{mod}\ p)$ (see Theorem~\ref{thm:gen-fun-stab-coh-SymOmega}). We may therefore assume that $p|a+b$, and we have to show that $H_{a,b}(t)\neq 0$. Writing $a=pa'-i$, $b=pb'+i$ for some $0\leq i\leq p-2$, it suffices by Theorem~\ref{thm:recursion-hooks}(4), (6) to show that $H_{pa'-p+1,0}(t)\neq 0$ which follows again from Theorem~\ref{thm:gen-fun-stab-coh-SymOmega}.

\subsection{Duality Between Hooks and 2-Column Partitions}\label{subsec:duality-hook-2column}

The goal of this section is to describe the stable cohomology for $\bb{S}_{\ll}\Omega$ when $\ll=(m,d)'$ is a partition with two columns. We do so by relating it to the stable cohomology for an associated hook partition, which can then be calculated using Theorem \ref{thm:recursion-hooks}.

\begin{theorem}\label{thm:dualityFor2ColAndHooks}
    If $\lambda=(m,d)'$ is a partition with two columns as in Section~\ref{subsec:two-cols}, then
    $$H^i_{st} (\bbs_{\lambda} (\Omega) ) = H^{2m+1 - i}_{st} (\bbs_{(d+1,1^{m-d})} (\Omega)) \quad \text{for all} \ i.$$ 
\end{theorem}

\begin{proof} Consider a power $p^k$ of $p=\chr(\kk)$ with $p^k>d$ and $p^k>2m+2$. We have for all $i$
\[ 
\begin{aligned}
H^i_{st}(\bb{S}_{\ll}\Omega) &\overset{\text{Thm.~\ref{thm:two-columns}(2)}}{=} H_{i-m}\left(C_{\bullet}(1^d,-m-d-1)\right) \overset{\text{Prop.~\ref{prop:props-PCw}(4)}}{=} H_{i-m}\left(C_{\bullet}(1^d,p^k-m-d-1)\right)\\
&\overset{\text{Thm.~\ref{thm:stab-coh-ribbon=hom-Cw}}}{=}H^{p^k-1-i}_{st}\left(\bb{S}_{(d+1,1^{p^k-m-d-2})}\Omega\right).
\end{aligned}
\]
It follows that for the cohomology polynomials we get
\[h_{\bb{S}_{\ll}\Omega}(t) = t^{p^k-1}\cdot H_{d+1,p^k-m-d-2}(t^{-1}).\]
Applying Corollary~\ref{cor:strange-symm-Hab} with $A=d+1$ and $B=m-d$, it follows from the above equality that
\[h_{\bb{S}_{\ll}\Omega}(t) = t^{p^k-1}\cdot t^{2m+2-p^k} \cdot H_{d+1,m-d}(t^{-1}),\]
which is equivalent to the desired identifications of stable cohomology groups.
\end{proof}

\begin{example}
    Let $\ll=(6,3)'$, and suppose that $\chr(\kk)=3$. It follows from Theorem \ref{thm:dualityFor2ColAndHooks} that there is an equality of cohomology polynomials
    $$h_{\bbs_{\lambda} (\Omega)} (t) = t^{13} \cdot h_{\bbs_{(4,1^3)} (\Omega)} (t^{-1}).$$
    By the recursion of Theorem \ref{thm:recursion-hooks}, there is an equality
    $$t^{13}\cdot h_{\bbs_{(4,1^3)} (\Omega)} (t^{-1}) = t^{13} \cdot t^{-3} \left( t^{-2} + t^{-3} \right) = t^7 + t^8.$$
    It follows that the only nonvanishing stable cohomology for $\bb{S}_{\ll}\Omega$ is
    $$H^7_{st} (\bbs_{\lambda} (\Omega) ) = H^8_{st} (\bbs_{\lambda} (\Omega)) = \kk.$$
\end{example}

\section{Vanishing for Koszul modules}\label{sec:vanish-Koszul}

Koszul modules are finitely generated graded modules over a polynomial ring that have been studied
prominently in topology in the form of \defi{Alexander invariants of groups} \cites{PS-chen,PS-Artin,sch-suc-BGG,DPS-jump-loci,PS-johnson,PS-resonance}. More recently, they have emerged as fundamental objects in algebraic geometry and commutative algebra, bearing numerous analogies with Koszul cohomology groups \cites{AFPRW,AFPRW-groups,rai-sam,AFPRW-resonance,AFRS-reduced}, and with a multitude of applications the most striking of which was a new proof of Voisin's theorem on the validity of Green's conjecture for general curves \cites{V02,V05}. The advantage of the theory of Koszul modules is that it allows one to extend Voisin's results to sufficiently positive characteristics in an effective fashion. After recalling the main definitions and the optimal vanishing result for Koszul modules in sufficiently positive characteristic, we prove an effective vanishing result for Koszul modules in all characteristics, and explain the sense in which it is best possible. 

As in the Introduction, we let $V=\kk^n$ and $K\subset\bw^2 V$, and define $W(V,K)$ as the homology of \eqref{eqn:W}. Papadima and Suciu show in \cite{PS-resonance}*{Lemma~2.4} that the set-theoretic support of $W(V,K)$ is given by the \defi{resonance variety}
\begin{equation}\label{eq:resonance}
\mc{R}(V,K)=\Bigl\{a\in V^\vee  \, | \, \mbox{ there exists }b\in V^\vee \mbox{ such that } a\wedge b\in K^\perp\setminus \{0\} \Bigr\}\cup \{0\},
\end{equation}
where $K^{\perp}$ is the vector space of forms in $\bw^2 V^{\vee}$ vanishing identically on $K$. In particular, we have that $\mc{R}(V,K) = \{0\}$ if and only if $W_j(V,K) = 0$ for $j\gg 0$, and \cite{PS-resonance} asks for an effective bound for when the vanishing $W_j(V,K)=0$ starts. It is shown in \cites{AFPRW,AFPRW-groups} that if $\chr(\kk)=0$ or $\chr(\kk)\geq n-2$ then we have the equivalence
\begin{equation}\label{eq:thm-main-equiv}
\mc{R}(V,K) = \{0\} \Longleftrightarrow W_j(V,K) = 0\mbox{ for }j\geq n-3.
\end{equation}
Moreover, if $\dim(K)=2n-3$ and $n\geq 4$ then $W_{n-4}(V,K)\neq 0$. Using our earlier results, we can now prove a uniform result in all characteristics.

\begin{theorem}\label{thm:vanishing-koszul}
    Let $\kk$ be any field, and suppose that $\mc{R}(V,K) = \{0\}$. We have:
    \begin{itemize}
     \item[(i)] $W_j(V,K) = 0\mbox{ for }j\geq 2n-7$.
     \item[(ii)] If $\dim(K)=2n-3$ then 
     \[ W_{2n-8}(V,K) \neq 0 \Longleftrightarrow n = 3 + p^k\text{ for some }k,\text{ where }p=\chr(\kk). \]
     Moreover, if the equivalent statements above hold, then $W_{2n-8}(V,K)\simeq\kk$ is one-dimensional.
    \end{itemize}
\end{theorem}

Before explaining the proof of Theorem~\ref{thm:vanishing-koszul}, we illustrate it with an example.

\begin{example}\label{ex:vanishing-Koszul}
We assume throughout that $\dim(V)=n$, $\dim(K)=2n-3$, and that $W(V,K)$ has finite length. If $\chr(\kk)=0$ or $\chr(\kk)\geq n-2$, then $\dim(W_j(V,K))$ is uniquely determined, and is given for small values of the parameters $j,n$ in the following table (see \eqref{eq:thm-main-equiv} and also \cite{AFPRW}*{Theorem~1.4}):

\[\begin{array}{c|ccccc}
     {}_n\ \backslash^j&0&1&2&3&4 \\ \hline
     4&1&-&-&-&-\\
     5&3&5&-&-&- \\
     6&6&16&21&-&-\\
     7&10&35&70&84&-\\
     8&15&64&162&288&330\\
\end{array}\]
The equivalence~\eqref{eq:thm-main-equiv} is no longer true in small characteristics, for instance for $n=6$ and $\chr(\kk)=3$ one can have:
\[
\begin{array}{c|c|c|c|c|c|c|c}
     j & 0 & 1 & 2 & 3 & 4 & 5 & \cdots\\ \hline
     \dim W_j(V,K) &6&16&21&6&1 & 0 & \cdots \\
\end{array}
\]
Notice that $2n-7=5$ so Theorem~\ref{thm:vanishing-koszul}(i) is sharp. Moreover, $n=3+\chr(\kk)$ hence Theorem~\ref{thm:vanishing-koszul}(ii) applies.

Similarly, if $n=7$ and $\chr(\kk)=2$, one can have:
\[
\begin{array}{c|c|c|c|c|c|c|c|c|c}
     j & 0 & 1 & 2 & 3 & 4 & 5 & 6 & 7 & \cdots\\ \hline
     \dim W_j(V,K) &10&35&70&84&28&7&1 & 0 & \cdots \\
\end{array}
\]
which shows again the optimality of Theorem~\ref{thm:vanishing-koszul}.  If instead $n=8$ and $\chr(\kk)=2$, then the typical Hilbert function for $W(V,K)$ is
\[
\begin{array}{c|c|c|c|c|c|c|c|c|c}
     j & 0 & 1 & 2 & 3 & 4 & 5 & 6 & 7 & \cdots\\ \hline
     \dim W_j(V,K) &15&64&162&288&330&64&15 & 0 & \cdots \\
\end{array}
\]
so vanishing holds already for $j=2n-9$ (notice that $n\neq 3+2^k$ in this case).
\end{example}

\begin{proof}[Proof of Theorem~\ref{thm:vanishing-koszul}] Our argument follows closely the presentation in \cite{rai-Kmod}*{Proof of Theorem~2.2}. For part (i), we can replace $K$ by a generic subspace of dimension $2n-3$ (see \cite{AFPRW-groups}*{
Proof of Theorem 3.2} or \cite{rai-Kmod}*{Exercise 2.9}). We write $m=2n-3$, $\PP=\op{Proj}(S)\simeq\PP^{n-1}$, and form the exact Buchsbaum-Rim complex $\mc{B}_{\bullet}$ with $\mc{B}_0 = \Omega\oo\mc{O}_{\PP}(2)$, $\mc{B}_1 = K \oo \mc{O}_{\PP}$, 
\[\mc{B}_r = \bw^{n+r-2}K \oo \mc{O}_{\PP}(-n-2r+6) \oo \left(\Sym^{r-2}\Omega\right)^{\vee}\quad\text{for }r=2,\cdots,n-1.\]
The assertion that $W_j(V,K)=0$ is equivalent to the surjectivity of the induced map on global sections
\begin{equation}\label{eq:H0B1-to-H0B0} H^0(\PP,\mc{B}_1\oo\mc{O}_{\PP}(j)) \lra H^0(\PP,\mc{B}_0\oo\mc{O}_{\PP}(j)).
\end{equation}
We consider the hypercohomology spectral sequence
\[ E_1^{-r,s} = H^s(\PP,\mc{B}_r\oo\mc{O}_{\PP}(j)) \Longrightarrow \bb{H}^{s-r}(\PP,\mc{B}_{\bullet}\oo\mc{O}_{\PP}(j))=0,\]
where the vanishing of hypercohomology follows from the exactness of $\mc{B}_{\bullet}\oo\mc{O}_{\PP}(j)$. Since the cokernel of \eqref{eq:H0B1-to-H0B0} is $E_2^{0,0}$, in order to prove that \eqref{eq:H0B1-to-H0B0} is surjective it is enough to show that
\[E_1^{-r,r-1} = 0 \quad\text{for }r=2,\cdots,n-1. \]
Via Serre duality, this is further equivalent to showing that
\begin{equation}\label{eq:vanish-Hnr-Sym-twisted} H^{n-r}\left(\PP,(\Sym^{r-2}\Omega)\oo\mc{O}_{\PP}(2r-6-j)\right) = 0 \quad\text{for }r=2,\cdots,n-1.
\end{equation}
If $j=2n-7$ then $2r-6-j<0$ and $n-r<n-1$ for $2\leq r\leq n-1$, so the above vanishing follows from Theorem~\ref{thm:coh-omega}(3), proving part (i) of Theorem~\ref{thm:vanishing-koszul}.

If $j=2n-8$, then the same argument shows that \eqref{eq:vanish-Hnr-Sym-twisted} holds for $r<n-1$. More generally, we have
\[ E_1^{-r,s} = 0\quad\text{ for all }s>0\text{ and }r< n-1,\text{ and }\quad E_1^{-n+1,n-2} = H^{1}\left(\PP,\Sym^{n-3}\Omega\right)^{\vee}.\]
Since the spectral sequence converges to $0$, we must have that $E_2^{0,0}=W_{2n-8}(V,K)$ is isomorphic to $E_1^{-n+1,n-2}$. By Theorem~\ref{thm:gen-fun-stab-coh-SymOmega}, $H^1(\PP,\Sym^{n-3}\Omega)$ is non-zero if and only if $n-3=p^k$, in which case it is one-dimensional. This verifies part (ii) of the theorem and concludes our proof.
\end{proof}

\section*{Acknowledgements}
Experiments with Macaulay2 \cite{GS} and GAP \cite{doty-gap} have provided many valuable insights. Raicu acknowledges the support of the National Science Foundation Grants DMS-1901886 and DMS-2302341. VandeBogert acknowledges the support of the National Science Foundation Grant DMS-2202871. We are grateful to Henning Haahr-Andersen and Steven Sam for helpful discussions, and to the referees for their careful reading and valuable feedback.

\end{document}